\documentclass{amsart}
\usepackage{amsmath} 
\usepackage{amsfonts} 
\usepackage{amssymb} 
\usepackage[utf8]{inputenc} 
\usepackage[british]{babel} 
\usepackage{hyperref} 
\usepackage[protrusion=true,expansion=false]{microtype} 
\usepackage{comment} 
\usepackage{enumitem} 
\usepackage{mathtools}
\usepackage{amsthm} 
\usepackage{tikz}
\usetikzlibrary{positioning}
\usetikzlibrary{arrows.meta}
\usepackage{extarrows}
\usepackage{yfonts}
\usepackage{tikz-cd}
\usepackage{dsfont}
\usepackage{lmodern}
\usepackage{mathrsfs}
\usepackage{calc}
\usepackage{float} 
\usepackage{colortbl} 

\definecolor{dblue}{rgb}{0,0,0.7}
\definecolor{dgreen}{rgb}{0,0.6,0}
\definecolor{cgrey}{rgb}{0.8,0.8,0.85}

\hypersetup{
	unicode=true,
	colorlinks=true,
	citecolor=dgreen,
	linkcolor=dblue,
	anchorcolor=dblue
}

\newcommand{\concat}{\ensuremath{\mathbin{\raisebox{.9ex}{\scalebox{.7}{\(\frown\)}}}}}

\newcommand{\midcup}{\ensuremath{{\textstyle\bigcup}}}
\newcommand{\varep}{\ensuremath{\varepsilon}}
\newcommand{\vphi}{\ensuremath{\varphi}}

\newcommand{\res}{\ensuremath{\nobreak\mskip2mu\mathpunct{}\nonscript
  \mkern-\thinmuskip{\upharpoonright}\mskip5muplus1mu\relax}} 

\newcommand*{\defeq}{\mathrel{\vcenter{\baselineskip0.5ex \lineskiplimit0pt
                     \hbox{\scriptsize.}\hbox{\scriptsize.}}}%
                     =} 

\newcommand{\stringsymbol}{\ensuremath{\mathfrak{s}}} 
\newcommand{\sdim}{\ensuremath{\textup{d}_{\stringsymbol}}} 
\newcommand{\sideal}{\ensuremath{\mathcal{I}_{\stringsymbol}}}
\newcommand{\vc}{\ensuremath{\mathfrak{sd}}}

\newcommand{\calA}{\ensuremath{\mathcal{A}}}
\newcommand{\calB}{\ensuremath{\mathcal{B}}}
\newcommand{\calC}{\ensuremath{\mathcal{C}}}
\newcommand{\calF}{\ensuremath{\mathcal{F}}}
\newcommand{\calI}{\ensuremath{\mathcal{I}}}
\newcommand{\calL}{\ensuremath{\mathcal{L}}}
\newcommand{\calM}{\ensuremath{\mathcal{M}}}
\newcommand{\calN}{\ensuremath{\mathcal{N}}}
\newcommand{\calS}{\ensuremath{\mathcal{S}}}

\newcommand{\bbB}{\ensuremath{\mathbb{B}}}
\newcommand{\bbN}{\ensuremath{\mathbb{N}}}
\newcommand{\bbP}{\ensuremath{\mathbb{P}}}
\newcommand{\bbQ}{\ensuremath{\mathbb{Q}}}
\newcommand{\bbR}{\ensuremath{\mathbb{R}}}

\newcommand{\ddbbQ}{\ensuremath{\dot{\mathbb{Q}}}}

\newcommand{\power}{\ensuremath{\mathscr{P}}}

\newcommand{\CH}{\ensuremath{\mathsf{CH}}}

\newcommand{\ZF}{\ensuremath{\mathsf{ZF}}}
\newcommand{\ZFC}{\ensuremath{\mathsf{ZFC}}}

\DeclareMathOperator{\Add}{Add}
\DeclareMathOperator{\cf}{cf}
\DeclareMathOperator{\dom}{dom}

\DeclareMathOperator{\NSNF}{NSNF}
\DeclareMathOperator{\ot}{ot}
\DeclareMathOperator{\supp}{supp}

\DeclareMathOperator{\add}{\mathbf{add}}
\DeclareMathOperator{\cof}{\mathbf{cof}}
\DeclareMathOperator{\cov}{\mathbf{cov}}
\DeclareMathOperator{\non}{\mathbf{non}}

\newcommand{\frakb}{\ensuremath{\mathfrak{b}}}
\newcommand{\frakc}{\ensuremath{\mathfrak{c}}}
\newcommand{\frakd}{\ensuremath{\mathfrak{d}}}
\newcommand{\frakP}{\ensuremath{\mathfrak{P}}}

\newcommand{\sdot}[1]{\ensuremath{\dot{#1}}}

\newcommand{\ddf}{\ensuremath{\sdot{f}}}

\newcommand{\ddg}{\ensuremath{\sdot{g}}}

\newcommand{\ddr}{\ensuremath{\sdot{r}}}
\newcommand{\ddp}{\ensuremath{\sdot{p}}}
\newcommand{\ddq}{\ensuremath{\sdot{q}}}
\newcommand{\ddx}{\ensuremath{\sdot{x}}}
\newcommand{\ddX}{\ensuremath{\sdot{X}}}

\newcommand{\ddY}{\ensuremath{\sdot{Y}}}

\newcommand{\abs}[1]{\ensuremath{\mathchoice
	{\left\lvert#1\right\rvert}
	{\lvert#1\rvert}
	{\lvert#1\rvert}
	{\lvert#1\rvert}
}}

\providecommand{\mid}{}
\newcommand{\SetSymbol}[1][]{\ensuremath{%
	\;#1\vert\;
	\allowbreak
	\mathopen{}
}}
\DeclarePairedDelimiterX{\Set}[1]{\{}{\}}{%
	\renewcommand\mid{\SetSymbol[\delimsize]}
	#1
}

\newcommand{\tup}[1]{\renewcommand{\mid}{\SetSymbol}\ensuremath{\langle#1\rangle}}

\newcommand{\onright}[2]{\makebox[\widthof{$#1$}][l]{$\mathrlap{#2}$}} 

\newlength{\myl} 
\newcommand{\pfill}{\phantom{\vc()}} 
\newcommand{\rfill}{\phantom{\vc()}} 

\newcommand{\0}{\ensuremath{\mathds{O}}}
\newcommand{\1}{\ensuremath{\mathds{1}}}
\newcommand{\comp}{\ensuremath{\mathrel{\|}}}
\newcommand{\forces}{\ensuremath{\mathrel{\Vdash}}}

\newcommand{\lomega}{\ensuremath{{{<}\omega}}} 
\newcommand{\ol}[1]{\ensuremath{\overline{#1}}}

\theoremstyle{plain}
\newtheorem{thm}{Theorem}[section]
\newtheorem*{thmstar}{Theorem}
\newtheorem{cor}[thm]{Corollary}
\newtheorem{prop}[thm]{Proposition}
\newtheorem{fact}[thm]{Fact}
\newtheorem{lem}[thm]{Lemma}
\newtheorem{claim}{Claim}[thm]

\theoremstyle{definition}
\newtheorem{defn}[thm]{Definition}
\newtheorem{qn}[thm]{Question}
\newtheorem*{rk}{Remark}
\newtheorem*{eg}{Example}

\makeatletter
\newenvironment{poc}[1][Proof of Claim]{\par
\pushQED{{\renewcommand{\qedsymbol}{\ensuremath{\dashv}}\qed}}%
\normalfont \topsep6\p@\@plus6\p@\relax
\trivlist
\item\relax
{\itshape
#1\@addpunct{.}}\hspace\labelsep\ignorespaces
}{%
\popQED\endtrivlist\@endpefalse
}
\makeatother

\title[String Dimension]{String Dimension: VC dimension\linebreak for infinite shattering}
\author{Calliope Ryan-Smith}

\email{c.Ryan-Smith@leeds.ac.uk}
\urladdr{https://academic.calliope.mx}

\address{School of Mathematics, University of Leeds, LS2 9JT, UK}

\date{28th February 2024}
\thanks{The author's work was financially supported by EPSRC via the Mathematical Sciences Doctoral Training Partnership [EP/W523860/1]. For the purpose of open access, the author has applied a Creative Commons Attribution (CC BY) licence to any Author Accepted Manuscript version arising from this submission. No data are associated with this article.}

\begin{document}

\begin{abstract}
In computer science, combinatorics, and model theory, the VC~dimension is a central notion underlying far-reaching topics such as error rate for decision rules, combinatorial measurements of classes of finite structures, and neo-stability theory. In all cases, it measures the capacity for a collection of sets \(\calF\subseteq\power(X)\) to \emph{shatter} subsets of \(X\). The VC~dimension of this class then takes values in \(\bbN\cup\Set{\infty}\). We extend this notion to an infinitary framework and use this to generate ideals on \(2^\kappa\) of families of bounded shattering. We explore the cardinals characteristics of ideals generated by this generalised VC~dimension, dubbed \emph{string dimension}, and present various consistency results.

We also introduce the \emph{finality} of forcing iteration. A \(\kappa\)-final iteration is one for which any sequences of ground model elements of length less than \(\kappa\) in the final model must have been introduced at an intermediate stage. This technique is often used for, say, controlling sets of real numbers when manipulating values of cardinal characteristics, and is often exhibited as a consequence of a chain condition. We demonstrate a precise characterisation of such notions of forcing as a generalisation of distributivity.
\end{abstract}

\maketitle

\section{Introduction}
VC~dimension and VC~classes arose from an inspection of machine learning algorithms in \cite{vapnik_chervonenkis_class_1964} and \cite{vapnik_chervonenkis_uniform_1968} by Vapnik and \v{C}hervonenkis, in which the authors show that for a particular class of decision rules the rate of possible error falls precisely into one of two camps: exponential, or polynomial. In the latter case, a constant \(d\) is attributed to the rules that describes the greatest value at which the error rate \(e(n)=2^n\) for all \(n<d\), after which point the error rate is bounded by \(n^d\). This constant would become known as the VC~dimension of the class of rules, and such classes would be known as VC~classes.

\begin{defn}[VC~class]
Let \(X\) be an infinite set and \(\calF\subseteq\power(X)\). For \(A\subseteq X\), let \({\calF\cap A=\Set{Y\cap A\mid Y\in\calF}}\). Then, for \(n<\omega\), let
\begin{equation*}
f_\calF(n)=\max\Set*{\abs{\calF\cap A}\mid A\in[X]^n}
\end{equation*}
We say that \(\calF\) is a \emph{VC~class} if there is a number \(d<\omega\) such that for all \(n>d\), \(f_\calF(n)<2^n\). The least such \(d\) for which this holds is the \emph{VC~dimension} of \(\calF\), and it turns out that for all \(n\geq d\), \(f_\calF(n)\leq n^d\) (see \cite[Theorem~1]{vapnik_chervonenkis_uniform_1968} or \cite[Chapter~5, Theorem~1.2]{van_den_dries_tame_1998}).\footnote{In fact one can be even more precise. Either \(f_\calF(d)=2^d\) or for all \(n\geq d\), \(f_\calF(n)\leq\sum_{i<d}\genfrac{(}{)}{0pt}{}{n}{i}\).}

We say that \(\calF\) \emph{shatters} \(A\) if for all \(B\subseteq A\) there is \(Y\in\calF\) such that \(Y\cap A=B\). Equivalently, \(\calF\cap A=\power(A)\). Note then that \(f_\calF(n)=2^n\) if and only if \(\calF\) shatters some set of cardinality \(n\).
\end{defn}

The concept of VC~classes of sets proved incredibly robust and came to heavily influence classification theory in model theory. Instead of considering arbitrary collections of sets \(\calF\subseteq\power(X)\) one takes a first-order structure \(M\) and a first-order formula \(\vphi(x;y)\) in the language of \(M\). From these one considers the class \({M[\vphi]=\Set{\vphi(M;a)\mid a\in M^y}}\) of solutions to \(\vphi(x;a)\) as \(a\) varies over \(M^y\). This manifests as a subset of \(\power(M^x)\) and, hence, \(M[\vphi]\) may or may not be a VC~class. A formula \(\vphi(x;y)\) is then said to have the \emph{independence property} if \(M[\vphi]\) is not a VC~class, whereas \(\vphi(x;y)\) is \emph{NIP} (\emph{not the independence property}) otherwise. If \(M[\vphi]\) is NIP for all formulae \(\vphi\), then \(M\) itself is said to be NIP.

Suppose that \(M[\vphi]\) is not a VC~class. Then for all \(n<\omega\), there is a collection of \(x\)\nobreakdash-tuples \(\tup{a_i\mid i<n}\) and a collection of \(y\)\nobreakdash-tuples \(\tup{b_I\mid I\subseteq n}\) such that \(M\vDash\vphi(a_i;b_I)\) if and only if \(i\in I\). Hence by the compactness theorem for first-order logic, if \(\vphi(x;y)\) has infinite VC~dimension, then for all infinite cardinals \(\kappa\) there is an elementary extension \(N\succeq M\) and sequences \(\tup{a_\alpha\mid\alpha<\kappa}\), \(\tup{b_I\mid I\subseteq\kappa}\) such that \(N\vDash\vphi(a_\alpha;b_I)\) if and only if \(\alpha\in I\). It is therefore immaterial in the study of first-order theories to distinguish between various `sizes' of infinite VC~dimension.

This conceptualisation of VC~dimension through structures and formulae need not be restricted to first-order logic, and indeed it is not. For example, the notion of VC~dimension for positive logic has been developed in \cite{dobrowolski_amalgamation_2023}, and VC~dimension is applied to the generalised idea of definable sets in \cite{van_den_dries_tame_1998}, amongst others. However, these are derived from roots deeply embedded in finite combinatorics and their definitions exhibit this. In contexts of classes of structures that do not obey a compactness theorem, or pure classes of sets, it may be important to distinguish between various classes that are not VC~classes in the finitary sense, but do still fail to shatter in some way. This paper introduces a highly general form of VC~dimension, dubbed string dimension, and uses it to derive a cardinal characteristic \(\vc(\delta,\kappa)\)---where \(\kappa^+\leq\vc(\delta,\kappa)\leq2^\kappa\)---that measures how many classes of bounded string dimension are required to cover \(2^\kappa\). We then investigate the possible values of \(\vc(\delta,\kappa)\) for \(\kappa=\aleph_0\) or \(\aleph_1\), show the relations that can be built between them, and describe some possible constellations of these cardinals. General results concerning how Cohen forcing affects \(\vc(\delta,\kappa)\) for other values of \(\kappa\) are also described, as well as some more general forcing notions. In the case that \(\kappa\) is a strong limit (that is, for all \(\lambda<\kappa\), \(2^\lambda<\kappa\)) we obtain the following maximality result.

\begin{thm}
If \(\kappa\) is a strong limit and \(\delta\) is the least cardinal such that \(2^\delta\geq\cf(\kappa)\) then \(\vc(\delta,\kappa)=2^\kappa\). In particular, \(\vc(\aleph_0,\aleph_0)=2^{\aleph_0}\).
\end{thm}

We also classify a helpful feature of some forcing iterations, which we call \emph{finality}. When making forcing arguments for cardinal characteristics, one often uses an iteration or product forcing and then argues that, for example, no new real numbers are added by the notion of forcing in the final model. We generalise this behaviour for the sake of inspecting the effect of Cohen forcing on \(\vc(\delta,\kappa)\), providing it as an analogue to distributivity for iterations. This refines well-known techniques for proving this kind of quality that often use chain conditions, such as showing that if \(\bbP\) is the Cohen forcing that introduces \(\omega_1\)\nobreakdash-many new subsets of \(\omega\) to \(V\), then enumerating those new subsets as \(\tup{c_\alpha\mid\alpha<\omega_1}\) we get that
\begin{equation*}
(2^\omega)^{V[c_\alpha\mid\alpha<\omega_1]}=\bigcup_{\beta<\omega_1}(2^\omega)^{V[c_\alpha\mid\alpha<\beta]}
\end{equation*}
In line with this definition, we say that a forcing iteration \(\tup{\bbP_\alpha,\ddbbQ_\beta\mid\alpha\leq\gamma,\beta<\gamma}\) (or indeed a product forcing) is \emph{\(\kappa\)\nobreakdash-final} if for all \(V\)\nobreakdash-generic \(G\subseteq\bbP_\gamma\) and all \(\lambda<\kappa\),
\begin{equation*}
(V^\lambda)^{V[G]}=\bigcup_{\alpha<\gamma}(V^\lambda)^{V[G\res\alpha]}
\end{equation*}
In this case, the typical chain condition arguments about Cohen forcing manifest themselves more generally in the following proposition.
\begin{prop}
Let \(\bbP=\prod_{\alpha<\gamma}\bbP_\alpha\) be a product forcing of bounded support; that is, \(p\in\bbP\) only if there is \(\delta<\gamma\) such that \(p=p\res\delta\). If \(\bbP\) has the \(\cf(\gamma)\)\nobreakdash-chain condition then \(\bbP\) is \(\cf(\gamma)\)\nobreakdash-final.\qed
\end{prop}
This result is further extended to iterations in Proposition~\ref{prop:eg-cc}

\subsection{Structure of the paper}

In Section~\ref{s:preliminaries} we detail preliminaries for the paper and introduce an important change in notational convention. No knowledge of model theory is required to understand the results of the paper, but a basic understanding is required for the introduction to VC~dimension in Section~\ref{s:preliminaries;ss:model-theory;sss:vc} that motivates the definition of string dimension. This prerequisite understanding is supported by preliminaries on model theory in Section~\ref{s:preliminaries;ss:model-theory}. In Section~\ref{s:preliminaries;ss:cardinal-characteristics} we introduce the basics of cardinal characteristics, and in Section~\ref{s:preliminaries;ss:forcing} we touch on forcing. No knowledge of cardinal characteristics is necessary, but a passing familiarity with forcing is needed for Section~\ref{s:vc-forcing}.

In Section~\ref{s:str-dim} we introduce string dimension, the ideals that it generates, and many of its basic properties. In Section~\ref{s:vc-delta-kappa} we introduce the class of cardinal characteristics \(\vc(\delta,\kappa)\) and investigate some combinatorial bounds on their values. In Section~\ref{s:vc-forcing} we inspect how various notions of forcing affect \(\vc(\delta,\kappa)\), introducing the finality of a forcing and extending the New Set--New Function property introduced in \cite{cichon_combinatorial-b2_1993}. Later on we use these results to produce some constellations of values of \(\vc(\delta,\kappa)\) in the case that \(\kappa\) is \(\aleph_0\) or \(\aleph_1\). In Section~\ref{s:future} we summarise our consistency results and end with open questions.

\section{Preliminaries}\label{s:preliminaries}

We work in \(\ZFC\). We denote by \(\power(X)\) the power set of \(X\). Given a function \(f\colon X\to Y\) and a set \(A\subseteq X\), we denote by \(f\res A\) the restriction of \(f\) to \(A\), and by \(f``A\) the pointwise image \(\Set{f(a)\mid a\in A}\). When \(F\) is a set of functions we take \(F\res A\) to be the pointwise restriction \(\Set{f\res A\mid f\in F}\). Using the usual correspondence between \(\power(X)\) and \(2^X\) (that is, the bijection \(2^X\to\power(X)\) given by \(f\mapsto f^{-1}(1)\)) note that if \(F\subseteq 2^X\) corresponds to \(\calF\subseteq\power(X)\), then \(F\res A\) corresponds to \(\Set{Y\cap A\mid Y\in\calF}\). For sets \(X\) and \(Y\), we denote by \([X]^Y\) the set of all subsets of \(X\) of cardinality \(\abs{Y}\). We shall take ordinals to be represented set-theoretically by their von\nobreakdash-Neumann interpretations. In particular, we shall use \(n\) to mean the set \(\Set{0,\dotsc,n-1}\) without further comment.

Throughout, we shall generally use Greek letters to denote ordinals and cardinals rather than arbitrary sets. Unless otherwise specified, the letters \(\alpha\), \(\beta\), and \(\gamma\) will be used to denote ordinals. Given ordinals \(\alpha\) and \(\beta\), we shall denote by \(\alpha\oplus\beta\) the ordinal addition of \(\alpha\) and \(\beta\), and by \(\alpha\otimes\beta\) the ordinal product. If \(\alpha\) and \(\beta\) are also cardinals, then we shall denote by \(\alpha+\beta\) the cardinal addition of \(\alpha\) and \(\beta\), and by \(\alpha\times\beta\) the cardinal product. We shall use \(\sum\) only to refer to ordinal addition. \emph{We shall not use ordinal exponentiation.} Therefore, whenever \(\alpha\) and \(\beta\) are cardinals, \(\alpha^\beta\) refers to \(\abs{\Set{f\mid f\colon\beta\to\alpha}}\) (or indeed the set of functions \(\beta\to\alpha\)).

Given a set \(A\) of ordinals we shall denote by \(\ot(A)\) the \emph{order type} of \(A\). When \(\alpha\) and \(\beta\) are ordinals, we shall denote by \([\alpha]^{(\beta)}\) the set of subsets of \(\alpha\) of order type \(\beta\). Note that this is distinct from \([\alpha]^\beta\), the set of subsets of \(\alpha\) of \emph{cardinality} \(\abs{\beta}\).

\subsection{Model theory}\label{s:preliminaries;ss:model-theory}

When we denote a formula by \(\vphi(x)\) or \(\vphi(x;y)\), the variables \({x,\ y}\), etc.\ may be taken to be tuples of variables and we assume that these constitute all free variables appearing in \(\vphi\). If \(M\) is a first-order structure and \(x\) is a variable or tuple from \(M\), then \(M^x\) denotes the set of all tuples from \(M\) with the same length as \(x\). Given two linearly ordered sets \(I\) and \(J\) (usually finite or well-ordered tuples), we denote by \(I\concat J\) their concatenation. If \(\vphi(x;y)\) is an \(\calL\)\nobreakdash-formula, \(A\) is a set of \(x\)\nobreakdash-tuples, and \(b\in M^y\), then we define \(\vphi(A;b)\) to be the set \(\Set{a\in A\mid M\vDash\vphi(a;b)}\). Similarly we define \(M[\vphi]=\Set{\vphi(M^x;a)\mid a\in M^y}\).

\subsubsection{VC~dimension}\label{s:preliminaries;ss:model-theory;sss:vc}

VC~dimension as a classification tool in model theory is entirely standard, and we follow those standards here. Throughout this section we shall fix a first-order language \(\calL\) and an \(\calL\)\nobreakdash-structure \(M\).
\begin{defn}[VC~dimension]
Let \(\vphi(x;y)\) be an \(\calL\)\nobreakdash-formula (possibly with parameters from \(M\)). We say that \(\vphi(x;y)\) has \emph{VC~dimension \(\geq n\)} if there are sequences \(\tup{a_i\mid i<n}\) and \(\tup{b_I\mid I\subseteq n}\) from \(M^x\) and \(M^y\) respectively such that for all \(i<n\) and \(I\subseteq n\), \(M\vDash\vphi(a_i;b_I)\) if and only if \(i\in I\).

If there is a natural number \(n\) such that \(\vphi(x;y)\) has VC~dimension \(\geq n\), but does not have VC~dimension \(\geq n+1\), then we say that \(\vphi(x;y)\) has VC~dimension \(n\). Otherwise, we say that \(\vphi(x;y)\) has VC~dimension \(\infty\), or infinite VC~dimension.
\end{defn}

Note that the statement ``\(\vphi(x;y)\) has VC~dimension \(n\)'' is a first-order formula. In fact, if \(\vphi(x;y;z)\) is the formula \(\vphi\) without parameters, then for all \(c\in M^z\), if \(\vphi(x;y;c)\) has VC~dimension \(\geq n\), then so does \(\vphi(x;y\concat z)\). Therefore if \(M\) is such that no formula \(\vphi(x;y)\) has infinite VC~dimension then this is true for all elementarily equivalent models, and we say that \(M\) is \emph{NIP}, or \emph{dependent}.\footnote{In fact one needs a slightly stronger argument than this. See \cite{simon_guide_2015} for a thorough treatment of this topic.} By this same argument, being NIP is a property of a complete first-order theory.

Theories that are NIP have powerful structural decompositions that are explored in, for example, \cite{simon_guide_2015}, as well as many other places. A vital result obtained using VC~theory as a central tool was the solution to Pillay's conjecture for groups definable in \(o\)\nobreakdash-minimal structures (see \cite{hrushovski_groups_2008}). We shall not go into \(o\)-minimality, but a good resource for learning about it is \cite{van_den_dries_tame_1998}. The following is a direct quote from \cite{hrushovski_groups_2008}.

\begin{thmstar}[Hrushovski--Peterzil--Pillay]
If \(G\) is a definably compact group definable in a saturated \(o\)-minimal expansion of a real-closed field, then the quotient \(G/G^{00}\)of \(G\) by its smallest type-definable subgroup of bounded index \(G^{00}\) is, when equipped with the logic topology, a compact Lie group whose dimension (as a Lie group) equals the dimension of \(G\) (as a definable set in an \(o\)\nobreakdash-minimal structure).
\end{thmstar}

The statement brushes over that a group being definable in an NIP theory is enough to guarantee the existence of \(G^{00}\), the smallest type-definable subgroup of bounded index. Earlier results, such as Baldwin--Saxl from \cite{baldwin_logical_1976}, use NIP to allow such constructions in the first place.

\begin{thmstar}[Baldwin--Saxl]
Let \(G\) be a group definable in an NIP theory \(T\). Let \(H_a\) be a uniformly definable family of subgroups of \(G\). Then there is an integer \(N\) such that for any finite intersection \(\bigcap_{a\in A}H_a\), there is a subset \(A_0\in[A]^N\) with \(\bigcap_{a\in A}H_a=\bigcap_{a\in A_0}H_a\).\qed
\end{thmstar}

At the intersection of model theory and combinatorics the independence property is similarly being used to for regularity and partition results in the family of Erd\H{o}s\nobreakdash--Hajnal (for example in \cite{chernikov_regularity_2018}).

\subsection{Cardinal characteristics}\label{s:preliminaries;ss:cardinal-characteristics}

The field of cardinal characteristics can be considered to be the investigation of topological or combinatorial characteristics or invariants of ideals of sets, particularly ideals on Polish spaces or generalised Polish spaces. An \emph{ideal} on a set \(X\) is a collection \(\calI\subseteq\power(X)\) of subsets of \(X\) such that: \(\emptyset\in\calI\); if \(A,B\in\calI\) then \(A\cup B\in\calI\); and if \(A\subseteq B\in\calI\) then \(A\in\calI\).

An ideal is said to be \emph{\(\kappa\)\nobreakdash-complete} if for all \(\lambda<\kappa\) and collections \({\Set{A_\alpha\mid\alpha<\lambda}\subseteq\calI}\), \(\bigcup_{\alpha<\lambda}A_\alpha\in\calI\). We often use \emph{\(\sigma\)\nobreakdash-complete} to mean \(\aleph_1\)\nobreakdash-complete. That is, \(\calI\) is \(\sigma\)\nobreakdash-complete if it is closed under countable unions. By a \emph{\(\sigma\)-ideal} we mean a \(\sigma\)\nobreakdash-complete ideal. Given a collection \(\calA\subseteq\power(X)\), we shall denote by \(\tup{\calA}\) the ideal on \(X\) generated by \(\calA\):
\begin{equation*}
\tup{\calA}=\Set*{B\subseteq X\mid(\exists A_0,\,\dotsc,\,A_{n-1}\in\calA)\,B\subseteq\bigcup\Set*{A_i\mid i<n}}
\end{equation*}
Similarly, we shall denote by \(\tup{\calA}_\sigma\) the \(\sigma\)\nobreakdash-ideal on \(X\) generated by \(\calA\):
\begin{equation*}
\tup{\calA}_\sigma=\Set*{B\subseteq X\mid(\exists \calB\in[\calA]^{{\leq}\omega})\,B\subseteq\bigcup\calB}
\end{equation*}

Given a \(\sigma\)\nobreakdash-ideal \(\calI\) on an infinite set \(X\), we may define various combinatorial characteristics of the ideal in the following way.
\begin{alignat*}{2}
\cof(\calI)&=\min\{\abs{\calA}&&\mid\calA\subseteq\calI,\,(\forall B\in\calI)(\exists A\in\calA)B\subseteq A\}\tag{Cofinality}\\
\cov(\calI)&=\min\{\abs{\calA}&&\mid \calA\subseteq\calI,\,\midcup\calA=X\}\tag{Covering}\\
\non(\calI)&=\min\{\abs{A}&&\mid A\subseteq X,\,A\notin\calI\}\tag{Uniformity}\\
\add(\calI)&=\min\{\abs{\calA}&&\mid\calA\subseteq\calI,\,\midcup\calA\notin\calI\}\tag{Additivity}
\end{alignat*}

It is easy to prove that these cardinals form the structure of Figure~\ref{fig:characteristics-of-calI}, where \(A\to B\) indicates that \(\ZFC\vdash\abs{A}\leq\abs{B}\).

\begin{figure}[h]
\begin{tikzcd}
&\cov(\calI)\ar[r]&\cof(\calI)\ar[r]&2^{\abs{X}}\\
\aleph_1\ar[r]&\add(\calI)\ar[u]\ar[r]&\non(\calI)\ar[u]
\end{tikzcd}
\caption{The combinatorial cardinal characteristics of \(\calI\).}
\label{fig:characteristics-of-calI}
\end{figure}
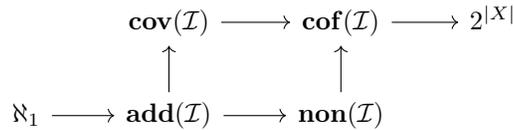

Traditionally one considers ideals on the set of real numbers (or subsets of other Polish spaces) with a basis of Borel sets, and so we have that \(\cof(\calI)\) is further bounded by \(\frakc=2^{\aleph_0}\). While the question ``Is \(\aleph_1<\frakc\)?'' is already independent of \(\ZFC\), it turns out that in many natural cases these characteristics vary between models of \(\ZFC\). For example, setting \(\calM\) to be the collection of meagre subsets of the real numbers, it is possible to have \(\cov(\calM)>\non(\calM)\) or \(\cov(\calM)<\non(\calM)\). Another traditional ideal to consider is \(\calN\), the collection of Lebesgue null sets of real numbers. Alongside two additional cardinal characteristics \(\frakb\) and \(\frakd\) (which are not relevant to the paper, but are defined and discussed in, for example, \cite{blass_combinatorial_2010}), we obtain Figure~\ref{fig:cichons-diagram}: Cicho\'n's diagram.\footnote{Astonishingly, Cicho\'n's diagram is known to be ``complete'', in that there are no provable (from \(\ZFC\)) inequalities between entries that are not already present in Figure~\ref{fig:cichons-diagram}. Furthermore, it is possible to construct models of \(\ZFC\) in which the ten ``independent'' entries (all but \(\add(\calM)\) and \(\cof(\calM)\)) are distinct from one another without the use of large cardinals (see \cite{goldstern_cichon_2022}).}

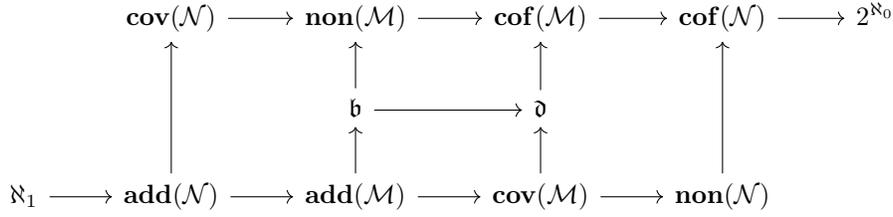
\begin{figure}[h]
\begin{tikzcd}
&\cov(\calN)\ar[r]&\non(\calM)\ar[r]&\cof(\calM)\ar[r]&\cof(\calN)\ar[r]&2^{\aleph_0}\\
&&\frakb\ar[u]\ar[r]&\frakd\ar[u]&\\
\aleph_1\ar[r]&\add(\calN)\ar[uu]\ar[r]&\add(\calM)\ar[u]\ar[r]&\cov(\calM)\ar[u]\ar[r]&\non(\calN)\ar[uu]&
\end{tikzcd}
\caption{Cicho\'n's diagram}
\begin{center}{\footnotesize The equalities \(\add(\calM)=\min\Set{\frakb,\cov(\calM)}\) and \(\cof(\calM)=\max\Set{\non(\calM),\frakd}\) are not labelled.}\end{center}
\label{fig:cichons-diagram}
\end{figure}

\begin{defn}[Tree]
For the purposes of this paper a \emph{tree} is a set \(T\subseteq2^{{<}\alpha}\), where \(\alpha\) is an ordinal, such that whenever \(t\in T\) and \(s\subseteq t\), \(s\in T\). A set \(b\subseteq T\) is a \emph{branch} if it is a \(\subseteq\)\nobreakdash-chain, and it is \emph{maximal} if for any branch \(c\subseteq T\), if \(b\subseteq c\) then \(b=c\). It is additionally \emph{cofinal} if it has order type \(\alpha\). Given a tree \(T\subseteq2^{{<}\alpha}\), we denote by \([T]\) its \emph{closure}, the set \(\Set{x\in2^\alpha\mid(\forall\beta<\alpha)x\res\beta\in T}\). We say that a set \(C\subseteq2^\alpha\) is \emph{closed} if and only if there is a tree \(T\subseteq2^{{<}\alpha}\) such that \(C=[T]\). We shall only deal with trees that have no non-cofinal maximal branches, and in this case we shall say that the tree is of \emph{height} \(\alpha\).
A tree \(T\) of height \(\omega\) is \emph{perfect} if for all \(s\in T\) there is \(t\supseteq s\) such that \(t\concat\tup{0}\in T\) and \(t\concat\tup{1}\in T\).
\end{defn}

The topology of trees is explored further in Section~\ref{s:str-dim;ss:topology}.

\begin{fact}
If \(T\) is perfect then \(\abs{[T]}=\frakc\). If \(C\subseteq2^\omega\) is closed and uncountable then there is a perfect tree \(T\) such that \([T]\subseteq C\). In particular, a closed subset of \(2^\omega\) is either countable or has cardinality \(\frakc\).
\end{fact}

\subsection{Forcing}\label{s:preliminaries;ss:forcing}

By a \emph{notion of forcing} (or just a \emph{forcing}) we mean a partially ordered set \(\tup{\bbP,\leq_{\bbP}}\) with a maximal element denoted \(\1_{\bbP}\). The subscript of both \(\leq_{\bbP}\) and \(\1_{\bbP}\) will be omitted when clear from context. We call the elements of \(\bbP\) \emph{conditions}, and for two conditions \(p,q\) we write \(q\leq p\) to mean that \(q\) is \emph{stronger} than \(p\) (or \(q\) \emph{extends} \(p\)). We follow Goldstern's alphabet convention, so \(p\) is never a stronger condition than \(q\), etc. We write \(p\comp p'\) to mean that \(p\) and \(p'\) are \emph{compatible}, that is there is \(q\in\bbP\) such that \(q\leq p\) and \(q\leq p'\). We denote by \(V^{\bbP}\) the collection of \(\bbP\)\nobreakdash-names.

A forcing is \emph{\(\kappa\)\nobreakdash-distributive} if for all \(\gamma<\kappa\) and all \(V\)\nobreakdash-generic \(G\), \((V^\gamma)^{V[G]}=(V^\gamma)^V\). A special case of distributivity is closure: A forcing is \emph{\(\kappa\)\nobreakdash-closed} if for all \(\gamma<\kappa\) and all descending chains \(\Set{p_\alpha\mid\alpha<\gamma}\), so \(\alpha<\beta\) implies that \(p_\alpha\geq p_\beta\), there is \(q\in\bbP\) such that for all \(\alpha<\gamma\), \(q\leq p_\alpha\). We shall write \(\sigma\)\nobreakdash-closed for \(\aleph_1\)\nobreakdash-closed.

A forcing \(\bbP\) has the \emph{\(\kappa\)\nobreakdash-chain condition}, or \(\kappa\)\nobreakdash-c.c., if for all \(A\subseteq\bbP\) with \(\abs{A}\geq\kappa\) there are distinct \(p,p'\in A\) such that \(p\comp p'\). Equivalently, every antichain in \(\bbP\) has cardinality less than \(\kappa\). We say \emph{countable chain condition}, written c.c.c., to mean \(\aleph_1\)\nobreakdash-c.c.

We shall make heavy use of a generalised form of Cohen's forcing: For cardinals \(\kappa,\lambda\), let \(\Add(\kappa,\lambda)\) be the forcing with conditions that are partial functions \({p\colon\lambda\times\kappa\to2}\) such that \(\abs{p}<\kappa\), ordered by \(q\leq p\) when \(q\supseteq p\). This forcing is \(\cf(\kappa)\)\nobreakdash-closed and is \((\kappa^{{<}\kappa})^+\)\nobreakdash-c.c. We may sometimes write \(\Add(X,Y)\), with \(X,Y\) not necessarily cardinals. In this case, we mean the notion of forcing with conditions that are partial functions \(p\colon Y\times X\to2\) such that \(\abs{p}<\abs{X}\), again ordered by reverse inclusion. Note that \(\Add(X,Y)\cong\Add(\abs{X},\abs{Y})\).

Given set collection of \(\bbP\)\nobreakdash-names \(\Set{\ddx_i\mid i\in I}\) in \(V\), we shall denote by \(\Set{\ddx_i\mid i\in I}^\bullet\) the name \(\Set{\tup{\1_{\bbP},\ddx_i}\mid i\in I}\). That is, \(\Set{\ddx_i\mid i\in I}^\bullet\) is a canonical choice of name for the set \(\Set{\ddx_i^G\mid i\in I}\) in \(V[G]\). This extends naturally to ground-model indexed tuples, functions, etc. We shall also sometimes use this bullet notation to describe a canonical name for an object. For example, \((2^{\check{\omega}_1})^\bullet\) is any name that \(\1_{\bbP}\) forces to be the set of all functions \(\check{\omega}_1\to2\) \emph{in the forcing extension}.

Given a set \(x\), we denote by \(\check{x}\) the \emph{check name} of \(x\), defined inductively as \({\Set{\check{y}\mid y\in x}^\bullet}\) so that \(\check{x}\) is always evaluated to the set \(x\) in the forcing extension. When it would be unwieldy to place a check above the symbol(s) representing a set, we may put the check to the right of the set instead, for example using \(\tup{1,2}\,\check{}\) for the check name of the tuple \(\tup{1,2}\), or \((2^{\omega_1})\,\check{}\) for the set of all \emph{ground model} functions \(\omega_1\to2\).

\begin{defn}[Product forcing]
Given a collection \(\Set{\tup{\bbP_i,\leq_i}\mid i\in I}\) of notions of forcing, we say that \(\tup{\bbP,\leq_\bbP}\) is a \emph{product forcing of \(\Set{\bbP_i\mid i\in I}\)} if:
\begin{itemize}
\item Each condition \(p\in\bbP\) is an element of the set-theoretic product \(\prod_{i\in I}\bbP_i\). That is, \(p\) is a function with domain \(I\) such that \(p(i)\in\bbP_i\) for all \(i\in I\). For \(p\in\bbP\) we define the \emph{support} of \(p\) to be \(\supp(p)=\Set{i\in I\mid p(i)\neq\1_{\bbP_i}}\).
\item For all \(p\in\prod_{i\in I}\bbP_i\) such that \(\supp(p)\) is finite, \(p\in\bbP\).
\item \(q\leq_{\bbP}p\) if and only if for all \(i\in I\), \(q(i)\leq_ip(i)\).
\item For all \(p,\,p'\in\bbP\), if \(\supp(p)\cap\supp(p')=\emptyset\) then the function \(q\) given by \(q(i)=\min\Set{p(i),p'(i)}\) is a condition in \(\bbP\).
\item For all \(p\in\bbP\) and \(I_0\subseteq I\) the function \(p\res I_0\) is a condition in \(\bbP\), where \(p\res I_0(i)=p(i)\) if \(i\in I_0\) and \(\1_{\bbP_i}\) otherwise.
\end{itemize}
We shall write \(\bbP=\prod_{i\in I}\bbP_i\) to mean that \(\bbP\) is a product forcing of \(\Set{\bbP_i\mid i\in I}\). Given \(\bbP=\prod_{i\in I}\bbP_i\) and \(I_0\subseteq I\), we denote by \(\bbP\res I_0\) the notion of forcing given by \(\Set{p\res I_0\mid p\in\bbP}=\Set{p\in\bbP\mid\supp(p)\subseteq I_0}\), with \(q\leq_{\bbP\res I_0}p\) if \(q\leq_{\bbP}p\). In particular, if \(I\) is an ordinal and \(\alpha\in I\) then \(\bbP\res\alpha\) refers to the forcing \(\bbP\res\Set{\beta\mid\beta<\alpha}\).

In some cases we may consider partial functions \(p\), in which case we identify ``\(i\notin\dom(p)\)'' with ``\(p(i)=\1_{\bbP_i}\)''. In light of this identification, note that \(\bbP\res I_0\) is isomorphic to a product forcing of \(\Set{\bbP_i\mid i\in I_0}\).

By the \emph{\(\lambda\)-support product} \(\prod_{i\in I}\bbP_i\) we mean the product forcing given by those functions \(p\) such that \(\abs{\supp(p)}<\lambda\). For families \(\Set{\bbP_\alpha\mid\alpha<\gamma}\) of notions of forcing indexed by an ordinal, we say that \(\bbP=\prod_{\alpha<\gamma}\bbP_\alpha\) is of \emph{bounded support} if for all \(p\in\bbP\) there is \(\alpha<\gamma\) such that \(p=p\res\alpha\).
\end{defn}

\begin{fact}\label{fact:product-restriction}
Let \(\bbP=\prod_{i\in I}\bbP\), and let \(I_0\subseteq I\). Suppose that \(\ddX,\ddY\) are \(\bbP\res I_0\)\nobreakdash-names, \(p\in\bbP\), and \(p\forces\ddX=\ddY\). Then \(p\res I_0\forces\ddX=\ddY\). The same is true of the formula \(\ddX\in\ddY\).
\end{fact}

\begin{defn}[Forcing iteration]
Let \(\bbP\) be a notion of forcing and \(\ddbbQ\) a \(\bbP\)\nobreakdash-name such that \(\1_{\bbP}\forces``\ddbbQ\) is a notion of forcing''. Then we denote by \(\bbP\ast\ddbbQ\) the notion of forcing with conditions of the form \(\tup{p,\ddq}\) with \(p\in\bbP\) and \(\1_{\bbP}\forces\ddq\in\ddbbQ\), and order given by \(\tup{p',\ddq'}\leq\tup{p,\ddq}\) if \(p'\leq p\) and \(p'\forces\ddq'\leq\ddq\). By Scott's trick, this can be represented as a set in the ground model, rather than a proper class of names.

We shall say that a collection \(\tup{\bbP_\alpha,\ddbbQ_\beta\mid\alpha\leq\gamma,\,\beta<\gamma}\) is a \emph{(forcing) iteration} if
\begin{enumerate}[label=\textup{\arabic*.}]
\item \(\bbP_0=\{\1\}\);
\item for all \(\alpha<\gamma\), \(\1_{\bbP_\alpha}\forces``\ddbbQ_\alpha\) is a notion of forcing'';
\item for all \(\alpha\leq\gamma\), the conditions in \(\bbP_\alpha\) are functions with domain \(\alpha\) such that, for all \(\beta<\alpha\), \(\1_{\bbP_\beta}\forces p(\beta)\in\ddbbQ_\beta\);
\item for all \(\alpha\leq\gamma\), the ordering on \(\bbP_\alpha\) is given by \(q\leq p\) if and only if, for all \(\beta<\alpha\), \(q\res\beta\leq_{\bbP_\beta}q\res\beta\), and \(q\res\beta\forces_{\bbP_\beta}q(\beta)\leq p(\beta)\); and
\item whenever \(\beta<\alpha\leq\gamma\), \(q\in\bbP_\beta\) and \(p\in\bbP_\alpha\) are such that \(q\leq_{\bbP_\beta}p\res\beta\), the condition \(q\cup(p\res(\alpha\backslash\beta))\in\bbP_\alpha\).
\end{enumerate}

Note that at successor stages we have \(\bbP_{\alpha+1}\cong\bbP_\alpha\ast\ddbbQ_\alpha\). If \(G\) is \(V\)\nobreakdash-generic for \(\bbP_\gamma\), and \(\alpha<\gamma\), we shall denote by \(G\res\alpha\) the restriction of each function \(p\in G\) to the domain \(\alpha\). This is a \(V\)\nobreakdash-generic filter for \(\bbP_\alpha\), and as such \(G\) provides a chain \(\tup{V[G\res\alpha]\mid\alpha\leq\gamma}\) of models of \(\ZFC\).

Given a condition \(p\in\bbP_\gamma\), and \(\alpha\leq\gamma\), we denote by \(p\res\alpha\) the condition in \(\bbP_\alpha\) given by that restriction. We shall denote by \(\bbP_\gamma/\alpha\) the canonical \(\bbP_\alpha\)\nobreakdash-name for the notion of forcing in \(V[G\res\alpha]\) given by the final \(\gamma\backslash\alpha\) iterands. Similarly, we shall denote by \(p/\alpha\) the canonical name for the \(\bbP_\gamma/\alpha\)\nobreakdash-condition that is the final \(\gamma\backslash\alpha\) co-ordinates of \(p\).
\begin{equation*}
p/\alpha=\Set*{\langle \1_{\bbP_\alpha},\langle\check{\delta},p(\delta)\rangle^\bullet\rangle\mid\alpha\leq\delta<\gamma}
\end{equation*}
This allows us to view \(\bbP_\gamma\) as the iteration \(\bbP_\alpha\ast\bbP_\gamma/\alpha\). As in the case of iterations, we have the following.
\end{defn}

\begin{fact}\label{fact:iteration-restriction}
Let \(\tup{\bbP_\alpha,\ddbbQ_\beta\mid\alpha\leq\gamma,\beta<\gamma}\) be an iteration and suppose that \(\ddX,\ddY\) are \(\bbP_\alpha\)\nobreakdash-names. For \(p\in\bbP_\gamma\), if \(p\forces\ddX=\ddY\) then \(p\res\alpha\forces\ddX=\ddY\). The same is true of the formula \(\ddX\in\ddY\).
\end{fact}

Note that any well-ordered product forcing \(\bbP=\prod_{\alpha<\gamma}\bbP_\alpha\) can be viewed as an iteration via \(\tup{\bbP\res\alpha,\ \check{\bbP}_\beta\mid\alpha\leq\gamma,\beta<\gamma}\), and hence many results regarding forcing iterations descend directly to products, such as Fact~\ref{fact:product-restriction} being a descent of Fact~\ref{fact:iteration-restriction}.

\section{String dimension}\label{s:str-dim}

In the context of VC~dimension in model theory, the formula \(\vphi(x;y)\) is said to \emph{shatter} a set \(A\subseteq M^x\) if there is a sequence \(\tup{b_I\mid I\subseteq A}\) of \(y\)\nobreakdash-tuples such that \(M\vDash\vphi(a;b_I)\) if and only if \(a\in I\). This is in practice identical to the definition of shattering used in the introduction: \(\vphi(x;y)\) shatters \(A\) if and only if \({\Set{\vphi(M^x;b)\mid b\in M^y}}\) shatters \(A\). Equivalently, if \(\Set{\vphi(A;b)\mid b\in M^y}=\power(A)\). Therefore \(\vphi\) has VC~dimension at least \(n\) in \(M\) if and only if \(\vphi\) shatters some set of \(x\)\nobreakdash-tuples of cardinality at least \(n\). This idea can easily be extended to infinite sets of \(x\)\nobreakdash-tuples, and even to pure sets devoid of structure.

\begin{defn}[Shattering and string dimension]
Let \(X\) be a set and \(F\subseteq2^X\). Then we say that \(F\) \emph{shatters} \(A\subseteq X\) if \(F\res A=2^A\). The \emph{string dimension} of \(F\), denoted \(\sdim(F)\), is the least cardinal \(\delta\) such that \(F\) shatters no \(A\in[X]^\delta\).
\end{defn}
\begin{rk}
Note that the definition of string dimension is almost the same as VC~dimension but has some important differences. Most obviously, string dimension takes into account the shattering of infinite sets and the cardinality of those sets.

Secondly, the string dimension measures the failure to shatter, whereas VC~dimension measures successful shattering. Hence, in the case that the VC~dimension of \(F\) is finite, say equal to \(n\), then \(\sdim(F)=n+1\). In the case that \(\sdim(F)=\delta\) for a limit cardinal \(\delta\), then \(F\) shatters sets of unbounded cardinality below \(\delta\), but no set of cardinality \(\delta\). This avoids the cumbersome notation \(\sdim(F)=``{<}\delta"\).

Finally, we have departed from considering families \(\calF\subseteq\power(X)\) and shall instead consider sets \(F\subseteq2^X\). This will greatly aid in notation clarity and succinctness for dealing with these objects in the future.
\end{rk}

\begin{prop}\label{prop:sdim-multiplicative}
Let \(X\) be a set and \(\lambda,\mu\) be cardinals. Suppose that for all \(\alpha<\mu\), \(F_\alpha\subseteq2^X\). If \(F=\bigcup_{\alpha<\mu}F_\alpha\) shatters a set of size \(\mu\times\lambda\) then one of the \(F_\alpha\) shatters a set of size \(\lambda\).
\end{prop}

\begin{proof}
Without loss of generality, we identify the set that \(F\) shatters with \(\mu\times\lambda\) and consider it a subset of \(X\). Suppose that for some \(\alpha<\mu\), \(F_\alpha\res\Set{\alpha}\times\lambda=2^{\Set{\alpha}\times\lambda}\). Then we are done. On the other hand, if this never happens then for all \(\alpha<\mu\) there is \(f_\alpha\colon\Set{\alpha}\times\lambda\to2\) such that \(f_\alpha\notin F_\alpha\res\Set{\alpha}\times\lambda\). However, \(F\) shatters \(\mu\times\lambda\), so there is \(g\in F\) such that \(g\res\mu\times\lambda=\bigcup_{\alpha<\mu}f_\alpha\). Then \(g\in F_\beta\), say, and \(g\res\Set{\beta}\times\lambda=f_\beta\), contradicting that \(f_\beta\notin F_\beta\res\Set{\beta}\times\lambda\).
\end{proof}

Proposition~\ref{prop:sdim-multiplicative} allows us to confidently consider ideals of sets of bounded string dimension.

\begin{defn}[String ideal]
Let \(\delta,\kappa\) be cardinals. The \emph{string ideal} \(\sideal(\delta,\kappa)\) is the \(\sigma\)\nobreakdash-ideal generated by the set
\begin{equation*}
\calS(\delta,\kappa)\defeq\Set*{F\subseteq2^\kappa\mid\sdim(F)<\delta}.
\end{equation*}
\end{defn}

\begin{prop}\label{prop:completion-of-sideal}
For all infinite cardinals \(\delta,\kappa\), \(\calS(\delta,\kappa)\) is \(\cf(\delta)\)\nobreakdash-complete. In particular, if \(\delta\) is infinite then \(\calS(\delta,\kappa)\) is an ideal and if \(\delta\) has uncountable cofinality then \(\calS(\delta,\kappa)=\sideal(\delta,\kappa)\).
\end{prop}

\begin{proof}
Let \(\mu<\cf(\delta)\) and \(\Set{F_\alpha\mid\alpha<\mu}\subseteq\calS(\delta,\kappa)\). We treat the case that \(\delta\) is a limit and a successor separately.

First, suppose that \(\delta\) is a successor, say \(\delta=\eta^+\), so \(\cf(\delta)=\delta\) and we are aiming to show that \(F=\bigcup\Set{F_\alpha\mid\alpha<\mu}\) shatters no set of cardinality \(\eta\). Since \(\delta>\mu\) we have that \(\eta\geq\mu\) and so, by Proposition~\ref{prop:sdim-multiplicative}, if \(F\) shatters a set of cardinality \(\eta=\mu\times\eta\) then there is \(\alpha<\mu\) such that \(F_\alpha\) shatters a set of cardinality \(\eta\), contradicting that \(F_\alpha\in\calS(\eta^+,\kappa)\).

On the other hand, suppose that \(\delta\) is a limit cardinal and let \(\tup{\delta_\beta\mid\beta<\cf(\delta)}\) be a cofinal sequence of cardinals in \(\delta\) such that \(\delta_0=\mu\). We wish to show that \({F=\bigcup\Set{F_\alpha\mid\alpha<\mu}}\) has string dimension less than \(\delta\), that is \(\sdim(F)\leq\delta_\beta\) for some \(\beta<\cf(\delta)\). Towards a contradiction, assume otherwise. By Proposition~\ref{prop:sdim-multiplicative}, for all \(\beta<\cf(\delta)\), \(F\) shatters a set of cardinality \(\delta_\beta=\delta_\beta\times\mu\) and so there is \(\alpha_\beta<\mu\) such that \(F_{\alpha_\beta}\) shatters a set of cardinality \(\delta_\beta\) as well. Since \(\mu<\cf(\delta)\) and \(\cf(\delta)\) is regular, there is \(\alpha<\mu\) such that for cofinally many \(\beta<\cf(\delta)\), \(F_\alpha\) shatters a set of cardinality \(\delta_\beta\). Hence, \(\sdim(F_\alpha)\geq\delta\), contradicting that \(F_\alpha\in\calS(\delta,\kappa)\).
\end{proof}

\subsection{Topology}\label{s:str-dim;ss:topology}

As we have seen, since the structure induced by shattering and dimension of subsets of \(2^X\) depend only on the cardinality of \(X\), we shall in general be only looking at \(2^\kappa\) for cardinals \(\kappa\). In this case, we will endow \(2^\kappa\) with the \emph{bounded-support product topology}. Namely, basic open subsets of \(2^\kappa\) will be of the form \([t]=\Set{x\in2^\kappa\mid x\supseteq t}\) where \(t\colon\alpha\to2\) for some \(\alpha<\kappa\) (though there is no harm in allowing \(\dom(t)\subseteq\alpha\) for some \(\alpha<\kappa\) instead). In this case, closed subsets of \(2^\kappa\) are precisely the sets of cofinal branches of subtrees of \(2^{{<}\kappa}\).

For \(F\subseteq2^\kappa\) we shall denote by \(\ol{F}\) the closure of \(F\) in this topology. Namely,
\begin{equation*}
\ol{F}=\Set*{x\in2^\kappa\mid(\forall\alpha<\kappa)(\exists y\in F)x\res\alpha=y\res\alpha}.
\end{equation*}

\begin{prop}\label{prop:sdim-closure}
Suppose that \(F\subseteq2^\kappa\) and \(\sdim(F)<\cf(\kappa)\). Then \(\sdim(\ol{F})=\sdim(F)\).
\end{prop}

\begin{proof}
We always have that \(\sdim(F)\leq\sdim(\ol{F})\) since \(F\subseteq\ol{F}\), so it suffices to prove the other direction. Indeed we shall actually prove the slightly stronger statement that whenever \(\ol{F}\) shatters a set \(A\subseteq\alpha<\kappa\) then \(F\) already shattered this set. However, this is clear from the definition of \(\ol{F}\) since for all \(f\in\ol{F}\res A\) there is \(g\in\ol{F}\) such that \(g\res A=f\), and since \(g\in\ol{F}\) there is \(g'\in F\) such that \(g'\res\alpha=g\res\alpha\), and so \(g'\res A=f\in F\res A\).
\end{proof}

In particular, this means that the ideal \(\sideal(\delta,\kappa)\) is generated by closed sets whenever \(\delta\leq\cf(\kappa)\) and indeed is closed under the topological closure operation. This fails in the case of \(\delta=\cf(\kappa)\) since, for example, the set
\begin{equation*}
\Set{f\colon\omega\to2\mid\abs{f^{-1}(1)}<\aleph_0}\subseteq2^\omega
\end{equation*}
has string dimension \(\aleph_0\), but its closure is the full space \(2^\omega\) which has string dimension \(\aleph_1\).

\begin{defn}
A set \(F\subseteq2^\kappa\) is \emph{nowhere dense} if for all \(s\in2^{{<}\kappa}\) there is \(t\in2^{{<}\kappa}\) such that \(t\supseteq s\) and \([t]\cap F=\emptyset\).
\end{defn}

\begin{prop}\label{prop:nwd}
If \(F\subseteq2^\kappa\) and \(\sdim(F)<\kappa\) then \(F\) is nowhere dense.
\end{prop}

\begin{proof}
Suppose that \(F\) is somewhere dense. That is, there is \(s\in2^{{<}\kappa}\) such that for all \(t\supseteq s\), \(F\cap[t]\neq\emptyset\). Let \(\delta<\kappa\) and set \(A=(\dom(s)\oplus\delta)\backslash\dom(s)\), noting that \(\dom(s)\oplus\delta<\kappa\) is bounded. For each \(y\in2^A\) there is \(x_y\in F\cap[s\concat y]\) and in this case \(x_y\res A=y\). This witnesses that \(F\res A=2^A\) and hence \(\sdim(F)\geq\delta\). Therefore \(\sdim(F)\geq\kappa\).
\end{proof}

\begin{rk}
While we are considering the bounded-support product topology so that closed sets are generated by trees, which gives us Proposition~\ref{prop:sdim-closure}, the proof of Proposition~\ref{prop:nwd} would have worked just as well in the \(\kappa\)\nobreakdash-support product topology in which basic open sets are of the form \([t]\) for partial \(t\colon\kappa\to2\) with \(\abs{\dom(t)}<\kappa\).
\end{rk}

\begin{cor}[Corollary of Proposition~\ref{prop:nwd}]
For all \(\delta\leq\kappa\), \(\sideal(\delta,\kappa)\subseteq\calM(\kappa,\kappa)\), the ideal of \(\kappa\)\nobreakdash-meagre subsets of \(2^\kappa\) (in either the bounded-support product or the \(\kappa\)\nobreakdash-support product topologies).\hfill\qed
\end{cor}

\begin{rk}
When \(\vc\) has been defined, this will show that \(\vc(\kappa,\kappa)\geq\cov(\calM(\kappa,\kappa))\).
\end{rk}

\begin{defn}
The \emph{(Lebesgue) null ideal} \(\calN\) is the \(\sigma\)\nobreakdash-ideal of null subsets of \(2^\omega\). It is generated as a \(\sigma\)\nobreakdash-ideal by the Lebesgue measure zero subsets of \(2^\omega\).
\end{defn}

\begin{prop}
\(\sideal(\aleph_0,\aleph_0)\subseteq\calN\).
\end{prop}

\begin{proof}
We shall require the following fact, \cite[Chapter~5, Theorem~1.2]{van_den_dries_tame_1998}, translated into our notation.

\begin{fact}[{\cite{van_den_dries_tame_1998}}]\label{fact:van-measure}
Let \(X\) be an infinite set and \(F\subseteq2^X\). Define
\begin{equation*}
f_F(n)\defeq\max(\abs{F\res A}\mid A\in[X]^n).
\end{equation*}
Then either \(f_F(n)=2^n\) for all \(n<\omega\), or there is \(d<\omega\) such that \(f_F(n)\leq n^d\) for all \(n\geq d\).{\renewcommand{\qedsymbol}{\ensuremath{\dashv}}\qed}
\end{fact}

Note that since \(\sideal(\aleph_0,\aleph_0)\) is generated by \(\calS(\aleph_0,\aleph_0)\) as a \(\sigma\)\nobreakdash-ideal, it suffices to prove that each \(F\in\calS(\aleph_0,\aleph_0)\) has Lebesgue measure zero. Let \(F\in\calS(\aleph_0,\aleph_0)\), and let \(N=\sdim(F)\). By Fact~\ref{fact:van-measure}, there is \(d<\omega\) such that for all sufficiently large \(n\), \(f_F(n)\leq n^d\). Hence, for all sufficiently large \(n\) there is \(O_n\subseteq2^n\) such that \(\abs{O_n}\leq n^d\) and for all \(x\in F\), \(x\res n\in O_n\). Therefore \(F\subseteq\bigcup\Set{[t]\mid t\in O_n}\) for all sufficiently large \(n\). The Lebesgue measure of \([t]\) is \(2^{-n}\), so the Lebesgue measure of \(F\) is at most than \(n^d\times2^{-n}\) for all \(n\). \(\lim_{n\to\infty}n^d2^{-n}=0\), so \(F\) must have Lebesgue measure zero.
\end{proof}

\section{Cardinal characteristics of \(\sideal(\delta,\kappa)\)}\label{s:vc-delta-kappa}

Recall that, unless specified otherwise, \(\alpha\), \(\beta\) and \(\gamma\) represent ordinals. Having defined the ideals \(\sideal(\delta,\kappa)\), it becomes natural to consider the cardinal characteristics associated with them. Recall the definition of the covering number of a \(\sigma\)\nobreakdash-ideal \(\calI\) on underlying set \(X\):
\begin{equation*}
\cov(\calI)\defeq\min\Set*{\abs{\calA}\mid\calA\subseteq\calI,\,\midcup\calA=X}
\end{equation*}

\begin{defn}\label{defn:vc-delta-kappa}
For infinite cardinals \(\delta\) and \(\kappa\), with \(\delta\leq\kappa^+\), we define the cardinal \(\vc(\delta,\kappa)\) to be \(\cov(\sideal(\delta,\kappa))\). That is, \(\vc(\delta,\kappa)\) is the smallest family of subsets of \(2^\kappa\) such that their union is \(2^\kappa\), but each subset in the family has string dimension less than \(\delta\).
\end{defn}

By noting that for \(\delta\leq\chi\leq\kappa^+\), \(\sideal(\delta,\kappa)\subseteq\sideal(\chi,\kappa)\) we immediately obtain that \(\vc(\delta,\kappa)\geq\vc(\chi,\kappa)\). Similarly, if \(\kappa\leq\eta\) then any covering of \(2^\eta\) by elements of \(\sideal(\delta,\eta)\) descends to a covering of \(2^\kappa\) by elements of \(\sideal(\delta,\kappa)\), and hence \(\vc(\delta,\kappa)\leq\vc(\delta,\eta)\).

In the specific case of \(\vc(\kappa^+,\kappa)\), note that if \(F\res X=2^X\) then \(\abs{F}\geq|2^X|\), and hence if \(\abs{F}<2^\kappa\), \(F\in\sideal(\kappa^+,\kappa)\).

Enumerate \(2^\kappa\) as \(\Set{x_\alpha\mid\alpha<2^\kappa}\), let \(\Set{\alpha_\gamma\mid\gamma<\cf(2^\kappa)}\) be a cofinal sequence in \(2^\kappa\), and set \(F_\gamma=\Set{x_\alpha\mid\alpha<\alpha_\gamma}\). Then \(\Set{F_\gamma\mid\gamma<\cf(2^\kappa)}\) is a covering of \(2^\kappa\) by elements of \(\sideal(\kappa^+,\kappa)\). Therefore \(\vc(\kappa^+,\kappa)\leq\cf(2^\kappa)\).

By Proposition~\ref{prop:completion-of-sideal}, \(\sideal(\kappa^+,\kappa)\) is \(\kappa^+\)\nobreakdash-complete, so \(\vc(\kappa^+,\kappa)\geq\kappa^+\) and indeed \(\vc(\delta,\kappa)\geq\kappa^+\) for all \(\delta\leq\kappa^+\). On the other end of the spectrum, for all \(x\in2^\kappa\) we have that \(\Set{x}\in\sideal(\aleph_0,\kappa)\), so \(\vc(\aleph_0,\kappa)\leq2^\kappa\) and indeed \(\vc(\delta,\kappa)\leq2^\kappa\) for all \(\delta\leq\kappa^+\).

Compiling all of these results, we obtain Figure~\ref{fig:constellation-full}, where \(A\to B\) indicates that \(\ZFC\vdash\abs{A}\leq\abs{B}\).

\begin{figure}[h]
\begin{tikzcd}[column sep=2em,row sep=1.7em,cells={nodes={align=center,text width=\myl}}] 
&&&\rfill\\
&&2^{\aleph_{\alpha+1}}\ar[ur,dashrightarrow]&\rfill\\
&2^{\aleph_\alpha}\ar[ur]&\vc(\aleph_0,\aleph_{\alpha+1})\ar[u]\ar[ur,dashrightarrow]&\rfill\\
\pfill\ar[ur,dashrightarrow]&\vc(\aleph_0,\aleph_\alpha)\ar[u]\ar[ur]&\vc(\aleph_1,\aleph_{\alpha+1})\ar[u]\ar[ur,dashrightarrow]&\rfill\\
\pfill\ar[ur,dashrightarrow]&\vc(\aleph_1,\aleph_\alpha)\ar[u]\ar[ur]&&\rfill\\
\pfill\ar[ur,dashrightarrow]&&\vc(\aleph_\alpha,\aleph_{\alpha+1})\ar[uu,dashrightarrow]\ar[ur,dashrightarrow]&\rfill\\
&\vc(\aleph_\alpha,\aleph_\alpha)\ar[uu,dashrightarrow]\ar[ur]&\vc(\aleph_{\alpha+1},\aleph_{\alpha+1})\ar[u]\ar[ur,dashrightarrow]&\rfill\\
\pfill\ar[ur,dashrightarrow]&\vc(\aleph_{\alpha+1},\aleph_\alpha)^\dagger\ar[u]\ar[ur]&\vc(\aleph_{\alpha+2},\aleph_{\alpha+1})^\dagger\ar[u]\ar[ur,dashrightarrow]&\rfill\\
\pfill\ar[ur,dashrightarrow]&&\aleph_{\alpha+2}\ar[u]\ar[ur,dashrightarrow]&\rfill\\
&\aleph_{\alpha+1}\ar[uu]\ar[ur]\\
\pfill\ar[ur,dashrightarrow]
\end{tikzcd}

{\small \(\dagger\) Not pictured is that \(\vc(\kappa^+,\kappa)\leq\cf(2^\kappa)\leq2^\kappa\) for all \(\kappa\).}
\caption{Relationships between values of \(\vc(\delta,\kappa)\) proved by \(\ZFC\).}
\label{fig:constellation-full}
\end{figure}

\begin{rk}
In the definition of \(\vc(\delta,\kappa)\) we excluded the case that \(\kappa\) is finite as this reduces to a combinatorial problem with a determined finite solution. We have also excluded the case that \(\delta\) is finite as \(\vc(n,\kappa)=\vc(\aleph_0,\kappa)\) for all finite \(n\). Finally, we have excluded the case that \(\delta>\kappa^+\) since \(2^\kappa\in\sideal(\kappa^{++},\kappa)\) and hence \(\vc(\kappa^{++},\kappa)=1\).
\end{rk}

\begin{eg}
By Proposition~\ref{prop:sdim-closure}, \(\calS(\aleph_1,\aleph_0)\) is \(\sigma\)\nobreakdash-closed and thus equal to \(\sideal(\aleph_1,\aleph_0)\). Therefore \(\sideal(\aleph_1,\aleph_0)\) may be simply described as
\begin{equation*}
\Set*{F\subseteq2^\omega\mid(\forall X\in[\omega]^\omega)F\res X\neq2^X}.
\end{equation*}
This ideal is first defined in \cite{cichon_combinatorial-b2_1993} as \(\frakP_2\). Note that if \(F\subseteq2^\omega\) is such that \(\abs{F}<\frakc\), then certainly \(F\in\sideal(\aleph_1,\aleph_0)\). Therefore \(\non(\sideal(\aleph_1,\aleph_0))=\frakc\) and, as with every such ideal, \(\cov(\sideal(\aleph_1,\aleph_0))\leq\cf(\frakc)\) by the same argument as the start of Section~\ref{s:vc-delta-kappa}.

The cardinal \(\vc(\aleph_1,\aleph_0)\) is not determined by \(\ZFC\) (unlike \(\vc(\aleph_0,\aleph_0)\), among others, as we will see later). Indeed, as shown in \cite{cichon_combinatorial-b2_1993}, \(\vc(\aleph_1,\aleph_0)\) (what they call \(\cov(\frakP_2)\)) is consistently strictly larger than all traditional cardinal characteristics of the continuum.\footnote{Those appearing in Cicho\'n's Diagram (Figure~\ref{fig:cichons-diagram}) as well as \(\mathfrak{a,b,d,e,g,h,i,m,p,r,s}\), and \(\mathfrak{u}\). See {\cite{blass_combinatorial_2010}} for definitions.} The only upper bound is \(\vc(\aleph_1,\aleph_0)\leq\cof(\calN)^+\), and this upper bound cannot be improved.
\end{eg}

\subsection{\(\vc\) for strong limit cardinals}\label{s:vc-delta-kappa;ss:strong-lim}

In this section we will exhibit a general construction for a class of trees \(\tup{T_\alpha\mid\alpha\in\operatorname{Ord}}\) such that \(\abs{[T_\alpha]}=2^{\abs{\alpha}}\), and it is in some sense ``difficult'' to have large subsets of \([T_\alpha]\) that are of low string dimension. In the case that \(\alpha=\kappa\) is a strong limit cardinal (that is, for all \(\beta<\kappa\), \(2^{\abs{\beta}}<\kappa\)), we can show that the height of \(T_\kappa\) is \(\kappa\). In this case we obtain a maximality result for \(\vc(\delta,\kappa)\) for certain values of \(\delta\).

We begin by treating \(\vc(\aleph_0,\aleph_0)\) to show off the initial stages of the construction.

\begin{prop}\label{prop:vc00-is-continuum}
\(\vc(\aleph_0,\aleph_0)=\frakc\). In fact, there is a perfect set \(X\subseteq2^\omega\) such that \(\sideal(\aleph_0,\aleph_0)\cap\power(X)=[X]^\lomega\).
\end{prop}

\begin{proof}
We shall construct a perfect tree \(T\subseteq2^\lomega\) such that whenever \(A\subseteq[T]\) has cardinality at least \(2^n\), \(A\) shatters a set of size \(n\). Hence, if \(A\subseteq[T]\) is infinite then \(\sdim(A)\geq\omega\). Therefore any covering of \(2^\omega\) (and by extension of \([T]\)) by finite\nobreakdash-dimensional sets must be of cardinality at least \(\frakc\).

We shall construct trees \(T_n\) for \(n<\omega\) of finite height such that \(T_n\) has \(2^n\) cofinal branches all of the same length. It will be constructed in such a way that for all \(n>0\) and \(k<n\), if \(A\subseteq[T_n]\) has cardinality at least \(2^k\), then \(A\) shatters a set of size \(k\).

Let \(T_0=\emptyset\), so \(T_0\) has 1 cofinal branch as required. Given \(T_n\), enumerate the cofinal branches as \(\tup{b_j\mid j<2^n}\). For each \(m\leq n\), enumerate \([2^n\times2]^{2^m}\) as \({\Set{I_{k,m}\mid k<N_m}}\), where \(N_m=\abs{[2^n\times2]^{2^m}}\). For each \(k<N_m\), let \(\vphi_{k,m}\colon I_{k,m}\to2^k\) be an arbitrary bijection. Then, for each \(\tup{j,i}\in2^n\times2\) and \(k<N_m\), set \(\partial_{k,m}(j,i)\) to be \(\vphi_{k,m}(j,i)\) if \(\tup{j,i}\in I_{k,m}\) and \(0^k\) otherwise. Setting \(\prod\) to be concatenation here, we define \(T_{n+1}\) to be the tree with cofinal branches precisely of the form
\begin{equation*}
c_{j,i}=b_j\concat\tup{i}\concat\prod_{m\leq n}\left(\prod_{k<N_m}\partial_{k,m}(j,i)\right)
\end{equation*}
for \(\tup{j,i}\in2^n\times2\). Note that \(T_{n+1}\) has exactly \(2^{n+1}\) cofinal branches of equal finite height as desired. Since we split each cofinal branch in \(T_n\) once to produce \(T_{n+1}\), the heights of the \(T_n\) are unbounded as \(n\) goes to \(\omega\). Furthermore, by construction, whenever \(A\subseteq[T_{n+1}]\) has cardinality at least \(2^m\), we have that \(A\) shatters a set of size \(m\). Let \(T=T_\omega=\bigcup_{n<\omega}T_n\). Then whenever \(A\subseteq[T]\) has cardinality at least \(2^m\), there is a level \(n<\omega\) such that \(\abs{A\res n}\geq2^m\). Then these branches pass through \(T_{n+1}\) from the first splitting and so will shatter a set of size \(m\).
\end{proof}

This construction is the most easy to produce, but we can easily extend it to arbitrary ordinals.

\subsubsection{The general construction}

Recall our conventions regarding ordinal and cardinal arithmetic: \(\oplus\) and \(\otimes\) refer strictly to ordinal addition and multiplication, \(+\) and \(\times\) refer strictly to cardinal addition and multiplication, and we do not use ordinal exponentiation.

We shall inductively build a class of trees \(\tup{T_\alpha\mid\alpha\in\operatorname{Ord}}\) such that the following hold:
\begin{enumerate}[label=\textup{\arabic*.}]
\item The only maximal branches of \(T_\alpha\) are cofinal;
\item setting \(\lambda_\alpha\) to be the cardinality of \([T_\alpha]\), the set of cofinal branches of \(T_\alpha\), we have \(\lambda_\alpha=2^{\abs{\alpha}}\);
\item setting \(\chi_\alpha\) to be the height of \(T_\alpha\), \(\alpha\leq\chi_\alpha<(2^{2^{\abs{\alpha}}})^+\); and
\item whenever \(A\subseteq[T_{\alpha+1}]\) is of cardinality at least \(2^{\abs{\beta}}\) for \(\beta\leq\alpha\), \(A\) shatters a set of size at least \(\abs{\beta}\).
\end{enumerate}
We begin with the same construction of \(T_n\) for \(n<\omega\), and let \({T_\omega=\bigcup\Set{T_n\mid n<\omega}}\). Given \(T_\alpha\), enumerate its cofinal branches as \(\tup{b_\gamma\mid\gamma<2^{\abs{\alpha}}}\) and, for \(\beta\leq\alpha\), enumerate \([2^{\abs{\alpha}}\times2]^{2^{\abs{\beta}}}\) as \(\Set{I_{\varep,\beta}\mid\varep<2^{2^{\abs{\alpha}}}}\). For each \(\varep<2^{2^{\abs{\alpha}}}\), let \(\vphi_{\varep,\beta}\colon I_\varep\to2^\beta\) be an arbitrary bijection. Then, for each \(\tup{\gamma,i}\in2^{\abs{\alpha}}\times2\) and \(\varep<2^{2^{\abs{\alpha}}}\), let \(\partial_{\varep,\beta}(\gamma,i)\) be \(\vphi_{\varep,\beta}(\gamma,i)\) if \(\tup{\gamma,i}\in I_{\varep,\beta}\) and \(0^\beta\) otherwise. Setting \(\prod\) to be concatenation once again, let \(T_{\alpha+1}\) be the tree with cofinal branches precisely of the form
\begin{equation*}
c_{\gamma,i}=b_\gamma\concat\tup{i}\concat\prod_{\beta\leq\alpha}\left(\prod_{\varep<2^{2^{\abs{\alpha}}}}\partial_{\varep,\beta}(\gamma,i)\right)
\end{equation*}
for each \(\tup{\gamma,i}\in2^{\abs{\alpha}}\times2\). Note that \(\lambda_{\alpha+1}=2^{\abs{\alpha+1}}\) as required. Given that \(\chi_\alpha\geq\alpha\) and each cofinal branch of \(T_{\alpha+1}\) is at least one larger than those in \(T_\alpha\), \(\chi_{\alpha+1}\geq\alpha+1\) as required. On the other hand,
\begin{equation*}
\chi_{\alpha+1}=\chi_\alpha\oplus1\oplus\left(\alpha\otimes2^{2^{\abs{\alpha}}}\otimes\alpha\right)
\end{equation*}
and since \(\abs{\chi_\alpha}\leq2^{2^{\abs{\alpha}}}\) and \((2^{2^{\abs{\alpha}}})^+\) is regular, \(\abs{\chi_{\alpha+1}}\leq2^{2^{\abs{\alpha+1}}}\) as required.

If \(\alpha>\omega\) is a limit cardinal, then let \(T_\alpha=\bigcup_{\beta<\alpha}T_\beta\).

\begin{defn}
Given a cardinal \(\kappa\), let \(\log(\kappa)\) be the least cardinal \(\lambda\) such that \(2^\lambda\geq\kappa\). In particular, \(\log(\kappa)\leq\kappa\), and \(\log(\kappa)=\kappa\) if and only if \(\kappa\) is a strong limit.
\end{defn}

\begin{thm}\label{thm:vc-strong-limit}
If \(\kappa\) is a strong limit, then \(\vc(\log(\cf(\kappa)),\kappa)=2^\kappa\).
\end{thm}

\begin{proof}
Using the notation of our calculations this section, we have that \(\chi_\kappa\geq\kappa\) and \({\chi_\kappa\leq\sup\Set{\chi_\alpha\mid\alpha<\kappa}}\). However, for all \(\alpha<\kappa\), \(\chi_\alpha<(2^{2^{\abs{\alpha}}})^+<\kappa\). Hence, \(\chi_\kappa=\kappa\), that is \(T_\kappa\subseteq2^{{<}\kappa}\) is a tree with \(2^\kappa\) cofinal branches. Suppose that \(A\subseteq[T_\kappa]\) is such that \(\abs{A}\geq\kappa\). For \(\delta<\log(\cf(\kappa))\), let \(A'\subseteq A\) be of cardinality \(2^\delta<\cf(\kappa)\), enumerated as \(\Set{x_\alpha\mid\alpha<2^\delta}\). Then for \(\alpha<\beta<2^\delta\), let \({\gamma_{\alpha,\beta}=\min\Set{\gamma<\kappa\mid x_\alpha(\gamma)\neq x_\beta(\gamma)}<\kappa}\). Since \(2^\delta<\cf(\kappa)\), \(\sup\Set{\gamma_{\alpha,\beta}\mid\alpha<\beta<2^\delta}\) is less than \(\kappa\), and so \(\abs{A'\res\gamma}=2^\delta\) for some \(\gamma<\kappa\). Suppose that \(\chi_\alpha<\gamma\leq\chi_{\alpha+1}\). Then by the construction of \(T_{\alpha+2}\), \(A'\) will shatter a set of size at least \(\delta\). Hence, for all \(\delta<\log(\cf(\kappa))\), \(A\) shatters a set of size \(\delta\). That is, \(\sdim(A)\geq\log(\cf(\kappa))\) and so \(A\notin\sideal(\log(\cf(\kappa)),\kappa)\). Therefore, to cover \([T_\kappa]\) in sets of dimension less than \(\log(\cf(\delta))\), each set will need to be of cardinality less than \(\kappa\). \(\cf(2^\kappa)\geq\kappa^+\), so \(2^\kappa\) such sets will be needed. Similarly, any covering of \(2^\kappa\) by sets of dimension less than \(\log(\cf(\kappa))\) will descend to a covering of \([T_\kappa]\), so the same restrictions apply.
\end{proof}

\begin{cor}\label{cor:vc-inaccessible}
If \(\kappa\) is strongly inaccessible, then \(\vc(\kappa,\kappa)=2^\kappa\) and, as a result, \(\vc(\delta,\kappa)=2^\kappa\) for all \(\delta\leq\kappa\).\qed
\end{cor}

\section{Forcing}\label{s:vc-forcing}

We wish to explore how notions of forcing will affect the value of \(\vc(\delta,\kappa)\). One robust method for reducing \(\vc(\delta,\kappa)\) is to construct a chain of forcing extensions \(\tup{M_\alpha\mid\alpha\leq\gamma}\) for some small \(\gamma\) such that, in \(M_\gamma\), \(\Set{(2^\kappa)^{M_\alpha}\mid\alpha<\gamma}\) is a covering of \(2^\kappa\) by elements of \(\sideal(\delta,\kappa)\) and hence \(\vc(\delta,\kappa)^{M_\gamma}\leq\cf(\gamma)\). This can be achieved if we meet the following criteria: For all \(\alpha<\gamma\), \((2^\kappa)^{M_\alpha}\in\sideal(\delta,\kappa)^{M_{\alpha+1}}\); for all \(\alpha<\gamma\), \(\sideal(\delta,\kappa)^{M_\alpha}\subseteq\sideal(\delta,\kappa)^{M_\gamma}\); and \({(2^\kappa)^{M_\gamma}=\bigcup\Set{(2^\kappa)^{M_\alpha}\mid\alpha<\gamma}}\).

We can also use similar concepts to increase the value of \(\vc(\delta,\kappa)\) in some scenarios, though the strategy is more difficult to heuristically describe. Instead of guaranteeing that \((2^\kappa)^{M_\alpha}\in\sideal(\delta,\kappa)^{M_\gamma}\), we enforce that any description of a subset of \(\sideal(\delta,\kappa)^{M_\gamma}\) of cardinality less than \(\tau\) is already present in an intermediate \(M_\alpha\), and then use \(M_{\alpha+1}\) to guarantee that this subset of \(\sideal(\delta,\kappa)^{M_\gamma}\) cannot cover \((2^\kappa)^{M_\gamma}\). In this case we must have that \(\vc(\delta,\kappa)^{M_\gamma}\) is at least \(\tau\), as any smaller candidate is destroyed.

In both these cases we are using variations of \emph{finality} and the \emph{New Set--New Function property}. The former guarantees that new functions \(\kappa\to2\) are not added to \(M_\gamma\), or that new subsets of \(\sideal(\delta,\kappa)\) are not added to \(M_\gamma\), and the latter (introduced in \cite{cichon_combinatorial-b2_1993}) controls if \((2^\kappa)^{M_\alpha}\in\sideal(\delta,\kappa)^{M_\gamma}\).

\subsection{Finality}\label{s:vc-forcing;ss:finality}

In forcing constructions it can be very helpful to be able to control when sequences of ordinals are added to a model. For example, one may hope that a product forcing \(\bbP=\prod_{\alpha<\gamma}\bbP_\alpha\) adds no new real numbers, or indeed any sequences of length \(\omega\), at limit stages. When this occurs, one can produce refined models of \(\ZFC\) in which cardinal characteristics are kept at precise values by, say, adding new real numbers at successor stages to eliminate old meagre or null sets. We shall generalise this concept and produce an exact criterion for it to hold.

\begin{defn}[Finality]
Let \(\tup{M_\alpha\mid\alpha\leq\gamma}\) be a chain of models of \(\ZF\) with the same ordinals, and let \(\kappa\) be such an ordinal. We say that the sequence is \emph{\(\kappa\)\nobreakdash-final} if for all \(\eta<\kappa\), and all \(b\colon\eta\to M_0\) in \(M_\gamma\), there is \(\alpha<\gamma\) such that \(b\in M_\alpha\). That is,
\begin{equation*}
\left(M_0^\eta\right)^{M_\gamma}=\bigcup\Set*{(M_0^\eta)^{M_\alpha}\mid\alpha<\gamma}.
\end{equation*}

An iteration \(\tup{\bbP_\alpha,\ddbbQ_\beta\mid\alpha\leq\gamma,\beta<\gamma}\) is \(\kappa\)\nobreakdash-final if for all \(V\)\nobreakdash-generic \(G\subseteq\bbP_\gamma\), \(\tup{V[G\res\alpha]\mid\alpha\leq\gamma}\) is \(\kappa\)\nobreakdash-final.
\end{defn}

In most cases, we will only care about the most basic utility of \(\kappa\)\nobreakdash-finality: That \(2^\kappa\), or perhaps \(\kappa^\kappa\), has no new elements added in the final limit of an iteration or product. However, it turns out that if \((\kappa\times\bbP)^\kappa\) has no new elements added at stage \(\gamma\), then \(\kappa^+\)\nobreakdash-finality is obtained for free.

\begin{defn}[Pseudodistributive]\label{defn:pseudodistributive}
Let \(\tup{\bbP_\alpha,\ddbbQ_\beta\mid\alpha\leq\gamma,\beta<\gamma}\) be a forcing iteration, let \(\bbP=\bbP_\gamma\), and let \(\calA=\Set{A_\alpha\mid\alpha<\eta}\) be a collection of maximal antichains in \(\bbP\). A \emph{pseudorefinement} of \(\calA\) is a maximal antichain \(A\) in \(\bbP\) such that for all \(q\in A\) there is \(\delta=\delta_q<\gamma\) such that for all \(\alpha<\eta\), \(p\in A_\alpha\), and \(r\leq p,q\), we have \(r\res\delta\forces q/\delta\leq p/\delta\).\footnote{Recall that \(p/\delta\) refers to the canonical \(\bbP_\delta\)\nobreakdash-name for the final \(\gamma\backslash\delta\) co-ordinates of \(\bbP\).} That is, whenever \(r_0\leq p\res\delta,q\res\delta\), either \(r_0\forces q/\delta\leq p/\delta\) or, for some \(s_0\leq r_0\), \(s_0\forces q/\delta\perp p/\delta\).

We shall say that \(\bbP\) is \emph{\(\kappa\)\nobreakdash-pseudodistributive} if for all \(\eta<\kappa\) and all collections \(\calA=\Set{A_\alpha\mid\alpha<\eta}\) of maximal antichains in \(\bbP\), there is a pseudorefinement for \(\calA\).
\end{defn}

For the rest of the section, \(\bbP\) shall always refer to the final stage of a forcing iteration \(\tup{\bbP_\alpha,\ddbbQ_\beta\mid\alpha\leq\gamma,\beta<\gamma}\) of length \(\gamma\), so \(\bbP=\bbP_\gamma\). We shall denote by \(\bbP/\delta\) the canonical \(\bbP_\delta\)\nobreakdash-name for the final \(\gamma\backslash\delta\) co-ordinates of \(\bbP\). If a chosen value of \(\delta\) is implicit, \(p_0\in\bbP_\delta\), and \(\1_{\bbP_\delta}\forces \ddp_1\in\bbP/\delta\), then we may denote by \(\tup{p_0,\ddp_1}\) the condition \(p\in\bbP\) such that \(p\res\delta=p_0\) and \(\1_{\bbP_\delta}\forces p/\delta=\ddp_1\), viewing \(\bbP\) as the iteration \(\bbP_\delta\ast\bbP/\delta\).

\begin{thm}\label{thm:finality-iff}
\(\bbP\) is \(\kappa\)\nobreakdash-final if and only if it is \(\kappa\)\nobreakdash-pseudodistributive.
\end{thm}

\begin{proof}
\((\implies)\).\quad{}Let \(\eta<\kappa\) and \(\calA=\Set{A_\alpha\mid\alpha<\eta}\) be a collection of maximal antichains in \(\bbP\). Let \(\ddf=\Set{\tup{p,\tup{\check{\alpha},\check{p}}^\bullet}\mid\alpha<\eta,p\in A_\alpha}\) so that for all \(V\)\nobreakdash-generic \(G\subseteq\bbP\), \(\ddf^G(\alpha)\) is the unique element of \(G\cap A_\alpha\). In particular, \(\1\forces\ddf\colon\check{\eta}\to\check{\bbP}\) and hence \(\1\forces(\exists\delta<\gamma)(\exists g\in V^{\bbP_\delta})\ddf=g\). Let \(A\) be a maximal antichain in \(\bbP\) such that \(q\in A\) only if there is \(\delta_q<\gamma\) and a \(\bbP_{\delta_q}\)\nobreakdash-name \(\ddg_q\) such that \(q\forces\ddg_q=\ddf\). We shall show that \(A\) is our desired pseudorefinement of \(\calA\).

Let \(q\in A\) and set \(\delta=\delta_q\). Suppose, for a contradiction, that there is \(\alpha<\eta\) and \(p\in A_\alpha\) such that \(q\comp p\) and \(r\leq p,q\) but \(r\res\delta\not\forces q/\delta\leq p/\delta\). By separativity, extending \(r\res\delta\) to \(r_0\) if necessary, let \(\ddr_1\) be a \(\bbP_\delta\)\nobreakdash-name such that
\begin{equation*}
r_0\forces``\ddr_1\leq q/\delta\text{ and }\ddr_1\perp p/\delta".
\end{equation*}
Let \(p'\in A_\alpha\) be such that \(p'\comp\tup{r_0,\ddr_1}\), witnessed by \(s\), and note that we must have \(p'\neq p\). Then \(s\leq\tup{r_0,\ddr_1}\leq\tup{q\res\delta,q/\delta}=q\), and so \(s\forces\ddf=\ddg_q\). However, \(s\leq p'\), so \(s\forces\ddf(\check{\alpha})=\check{p}'\). Hence, \(s\forces\ddg_q(\check{\alpha})=\check{p}'\) and in fact, since \(\ddg_q\) is a \(\bbP_\delta\)\nobreakdash-name, \(s\res\delta\forces\ddg_q(\check{\alpha})=\check{p}'\). On the other hand, \(r\leq p,q\), so \(r\forces\ddf(\check{\alpha})=\ddg_q(\check{\alpha})=\check{p}\) and so \(r\res\delta\forces\ddg_q(\check{\alpha})=\check{p}\), contradicting that \(s\res\delta\leq r\res\delta\).

\((\impliedby)\).\quad{}Let \(\eta<\kappa\) and \(\ddf\) be a \(\bbP\)\nobreakdash-name for a function \(\check{\eta}\to\check{X}\) for some \(X\in V\). For each \(\alpha<\eta\), let \(A_\alpha\) be a maximal antichain in \(\bbP\) such that \(p\in A_\alpha\) only if \(p\) decides the value of \(\ddf(\check{\alpha})\) and let \(A\) be a pseudorefinement of \(\Set{A_\alpha\mid\alpha<\eta}\). Fix \(q\in A\), set \(\delta=\delta_q\), and define the \(\bbP_\delta\)\nobreakdash-name \(\ddg_q\) as follows.
\begin{equation*}
\ddg_q=\Set*{\tup{r\res\delta,\tup{\check{\alpha},\check{x}}^\bullet}\mid p\in A_\alpha,\ p\forces\ddf(\check{\alpha})=\check{x},\ r\leq p,q}
\end{equation*}
We now need only show that \(q\forces\ddf=\ddg_q\) as then, by the predensity of \(A\), we will have that \(\1_{\bbP}\forces(\exists\delta<\gamma)(\exists g\in V^{\bbP_\delta})\ddf=g\) as required. Suppose that for some \(q'\leq q\) and \(x,x'\in X\), \(q'\forces\ddf(\check{\alpha})=\check{x}\land\ddg_q(\check{\alpha})=\check{x}'\).

Let \(p\in A_\alpha\) be such that \(p\comp q'\), witnessed by \(r\leq p,q'\), and note in particular that \(p\comp q\). Since \(p\) and \(q'\) are compatible, \(p\) decided the value of \(\ddf(\check{\alpha})\), and \(q'\) forced that \(\ddf(\check{\alpha})=\check{x}\), we must have that \(p\forces\ddf(\check{\alpha})=\check{x}\). Hence \(\tup{r\res\delta,\tup{\check{\alpha},\check{x}}^\bullet}\in\ddg_q\) and thus \(r\) (and indeed \(r\res\delta\)) forces that \(\ddg_q(\check{\alpha})=\check{x}\). On the other hand, \(r\leq q'\) so \(r\res\delta\forces\ddg_q(\check{\alpha})=\check{x}'\). To consider how it could be that \(r\res\delta\forces\ddg(\check{\alpha})=\check{x}\), suppose that \(p'\in A_\alpha\) is such that \(p'\forces\ddf(\check{\alpha})=\check{x}'\) and \(r'\leq p',q\) witnesses \(\tup{r'\res\delta,\tup{\check{\alpha},\check{x}'}^\bullet}\in\ddg_q\), with \(r\res\delta\comp r'\res\delta\). Let \(s_0\leq r\res\delta,r'\res\delta\). Then \(\tup{s_0,r'/\delta}\leq p',q\), so \(s_0\forces q/\delta\leq p'/\delta\). Similarly, \(\tup{s_0,r/\delta}\leq p,q\), so \(s_0\forces q/\delta\leq p/\delta\). Hence \(\tup{s_0,q/\delta}\leq p,p'\), so \(p=p'\) and \(x=x'\).
\end{proof}

\begin{rk}
In the proof of Theorem~\ref{thm:finality-iff} we implicitly use that the notion of forcing \(\bbP\) is \emph{separative}. That is, for all \(p,p'\in\bbP\), if \(p'\nleq p\) then there is \(q\leq p'\) such that \(q\perp p\). While one may always quotient a notion of forcing by its inseparable\footnote{Two conditions are inseparable if they witness that a preorder is not separative.} elements to produce a separative preorder, it is sometimes easier to allow inseparable elements. In this case, the proof works perfectly well by changing the condition in Definition~\ref{defn:pseudodistributive} to ``whenever \(r\leq p,q\), \(r\res\delta\) forces that \(q/\delta\leq p/\delta\) or there is an extension of \(r\res\delta\) forcing that \(q/\delta\) and \(p/\delta\) are inseparable''.
\end{rk}

\begin{rk}
Note that it only makes sense to consider the finality of an iteration of limit length, as \(\bbP\ast\ddbbQ\) is \(\kappa\)\nobreakdash-final if and only if \(\ddbbQ\) adds no functions \(\eta\to V\). That is, \(\1_{\bbP}\forces``\ddbbQ\text{ is }\check{\kappa}\text{\nobreakdash-distributive''}\). Henceforth we shall assume that any iteration is of limit length.
\end{rk}

In many cases we will be able to produce a pseudorefinement of a collection of maximal antichains in which the \(\delta\) is fixed. In this case, whenever there is a \(\bbP\)\nobreakdash-name for a function \(\eta\to V\) we will be able to determine in the ground model some upper bound \(\delta\) for where the function appears. A typical example of this is when considering bounded-support iterations with a chain condition, such as the following Proposition~\ref{prop:eg-cc}.

\begin{prop}\label{prop:eg-cc}
Let \(\tup{\bbP_\alpha,\ddbbQ_\beta\mid\alpha\leq\gamma,\,\beta<\gamma}\) be a forcing iteration such that \(\bbP=\bbP_\gamma\) has the \(\cf(\gamma)\)-chain condition and, for all \(p\in\bbP\), there is \(\alpha<\gamma\) such that \(\1_{\bbP}\forces\Set{\check{\beta}\mid p(\beta)\neq\1_{\ddbbQ_\beta}}^\bullet\subseteq\check{\alpha}\). Then \(\bbP\) is \(\cf(\gamma)\)-final.
\end{prop}

\begin{proof}
Let \(\eta<\cf(\gamma)\), and \(\Set{A_\alpha\mid\alpha<\eta}\) a collection of maximal antichains in \(\bbP\). By the chain condition, for each \(\alpha<\eta\) there is \(\delta_\alpha<\gamma\) such that for all \(p\in A_\alpha\) and \(\beta>\delta_\alpha\), \(\1_{\bbP}\forces p(\beta)=\1_{\ddbbQ_\beta}\). Furthermore, since \(\eta<\cf(\gamma)\), \({\delta=\sup\Set{\delta_\alpha\mid\alpha<\eta}<\gamma}\). Then \(\Set{\1_{\bbP}}\) is a pseudorefinement of \(\Set{A_\alpha\mid\alpha<\eta}\); if \(p\in A_\alpha\) then \(\1_{\bbP_\alpha}\forces p/\delta=\1_{\bbP}/\delta\), so certainly for all \(r\leq p,\1_{\bbP}\) we have \(r\forces p/\delta\leq\1_{\bbP}/\delta\).
\end{proof}

In the case of product forcing, the statement of Proposition~\ref{prop:eg-cc} becomes ``if \(\bbP=\prod_{\alpha<\gamma}\bbP_\alpha\) is of bounded support and has the \(\cf(\gamma)\)\nobreakdash-chain condition, then \(\bbP\) is \(\cf(\gamma)\)\nobreakdash-final''.

\begin{eg}
Another trivial instance of finality is when \(\bbP\) is distributive: If \(\bbP\) adds no functions \(\eta\to V\) for any \(\eta<\kappa\) then certainly \(\bbP\) is \(\kappa\)-final. However, chain conditions and distributivity are not the only ways to obtain finality.

\begin{lem}
For a forcing iteration \(\tup{\bbP_\alpha,\,\ddbbQ_\beta\mid\alpha\leq\gamma,\,\beta<\gamma}\), \(\bbP_\gamma\) is \(\kappa\)\nobreakdash-final if and only if, for all \(\alpha<\gamma\), \(\1_{\bbP_\alpha}\forces``\bbP_\gamma/\alpha\text{ is }\check{\kappa}\text{-final''}\).

In particular, if \(\bbP=\prod_{\alpha<\gamma}\bbP_\alpha\) is \(\kappa\)\nobreakdash-final and \(\bbR\) is any notion of forcing, then \(\bbR\times\bbP\), viewed as the product \(\bbR\times\prod_{\alpha<\gamma}\bbP_\alpha\), is also \(\kappa\)\nobreakdash-final.\qed
\end{lem}

Note that for any notions of forcing \(\bbP\) and \(\bbQ\), if \(\bbP\) is not \(\kappa\)\nobreakdash-c.c.\ then \(\bbP\times\bbQ\) is also not \(\kappa\)\nobreakdash-c.c. Similarly, if \(\bbP\) is not \(\kappa\)\nobreakdash-distributive then \(\bbP\times\bbQ\) is also not \(\kappa\)\nobreakdash-distributive. Hence, let \(\bbP\) be \(\Add(\omega_1,1)\), so \(\bbP\) is not c.c.c.\ but is \(\sigma\)-distributive, and let \(\bbQ\) be \(\Add(\omega,\omega_1)\), so \(\bbQ\) is c.c.c\ but is not \(\sigma\)-distributive. Since \(\bbQ\) is \(\omega_1\)\nobreakdash-final when viewed as the finite\nobreakdash-support product \(\prod_{\alpha<\omega_1}\Add(\omega,1)\) by Proposition~\ref{prop:add-gives-unions}, \(\bbP\times\bbQ\) is neither c.c.c.\ nor \(\sigma\)-distributive but is still \(\omega_1\)\nobreakdash-final when viewed as the finite\nobreakdash-support product \(\Add(\omega_1,1)\times\prod_{\alpha<\omega_1}\Add(\omega,1)\).

\end{eg}

Since we are looking at chains of models given by forcing, we actually get an even stronger conclusion than the definition of finality for free in the following Proposition~\ref{prop:strong-finality}.

\begin{prop}\label{prop:strong-finality}
Let \(\bbP\) be a forcing iteration of length \(\gamma\) that is \(\kappa\)\nobreakdash-final. Then for all \(V\)\nobreakdash-generic \(G\subseteq\bbP\), all \(\eta<\kappa\), and all \(\alpha<\gamma\),
\begin{equation*}
(V[G\res\alpha]^\eta)^{V[G]}=\bigcup\Set*{(V[G\res\alpha]^\eta)^{V[G\res\beta]}\mid\beta<\gamma}
\end{equation*}
That is, \(V[G]\) contains no functions \(\eta\to V[G\res\alpha]\) for any \(\alpha<\gamma\) that were not already present in some prior \(V[G\res\beta]\).
\end{prop}

\begin{proof}
Let \(f\colon\eta\to X\) be a function in \(V[G]\), where \(X\in V[G\res\alpha]\). Let \(\ddX\) be a \(\bbP_\alpha\)\nobreakdash-name for \(X\), and define the function \(g\) with domain \(\eta\) via demanding that \(g(\beta)\) is a \(\bbP_\alpha\)\nobreakdash-name for \(f(\beta)\). Then \(g\) is a function \(\eta\to V\), so there is \(\delta<\gamma\) such that \(g\in V[G\res\delta]\). Taking \(\delta\geq\alpha\) without loss of generality, we may determine \(f(\beta)\) in \(V[G\res\delta]\) by noting that \(f(\beta)=g(\beta)^{G\res\alpha}\).
\end{proof}

\subsection{Cohen forcing}\label{s:vc-forcing;ss:cohen-forcing}

Before we expand on the New Set--New Function property, it would behove us to show off how basic notions of forcing follow the constructions laid out at the beginning of Section~\ref{s:vc-forcing}. To this end, we shall spend some time exploring how Cohen forcing affects the value of \(\vc(\delta,\kappa)\), implementing our strategy in this case.

Recall that we define \(\Add(\chi,\lambda)\) as the notion of forcing with conditions that are partial functions \(p\colon\lambda\times\chi\to2\) such that \(\abs{p}<\chi\), ordered by \(q\leq p\) if and only if \(q\supseteq p\). If \(G\subseteq\Add(\chi,\lambda)\) is a generic filter, it must be of the form \({\Set{c\res A\mid A\in[\lambda\times\chi]^{{<}\chi}}}\) for some function \(c\colon\lambda\times\chi\to2\), and indeed \(c=\bigcup G\). Given a \(V\)\nobreakdash-generic filter \(G\subseteq\Add(\chi,\lambda)\), we shall denote by \(G\res\alpha\) the pointwise restriction \(\Set{p\res\alpha\times\chi\mid p\in G}\). This is a \(V\)\nobreakdash-generic filter for \(\Add(\chi,\alpha)\). We shall denote by \(G(\alpha)\) the pointwise restriction \(\Set{p\res\Set{\alpha}\times\chi\mid p\in G}\), which is a \(V\)\nobreakdash-generic filter for \({\Add(\chi,\Set{\alpha})\cong\Add(\chi,1)}\).

\begin{fact}
The following hold for \(\Add(\chi,\lambda)\) when \(\chi\) is regular:
\begin{enumerate}[noitemsep,label=\textup{\arabic*.}]
\item \(\Add(\chi,\lambda)\) is \(\chi\)\nobreakdash-closed, and so adds no sequences of ground-model elements of length less than \(\chi\). In particular, if \(\kappa<\chi\) then \(\Add(\chi,\lambda)\) adds no elements to \(2^\kappa\);
\item \(\Add(\chi,\lambda)\) is \((\chi^{{<}\chi})^+\)\nobreakdash-c.c.; and
\item \(\Add(\chi,\lambda)\) forces that \(\chi^{{<}\chi}=\chi\).
\end{enumerate}
\end{fact}

Let \(\lambda\), \(\chi\), and \(\chi'\) be cardinals, with \(\chi\) and \(\chi'\) regular. Since \(\Add(\chi,\lambda)\) is \(\chi\)\nobreakdash-closed and conditions in \(\Add(\chi,\lambda')\) are sequences of length less than \(\chi\), whenever \(G\) is \(V\)\nobreakdash-generic for \(\Add(\chi,\lambda)\) then \(\Add(\chi,\lambda')^V=\Add(\chi,\lambda')^{V[G]}\). Hence, one can view \(\Add(\chi,\lambda)\) as the \(\chi\)\nobreakdash-support product or iteration of \(\Add(\chi,1)\) with itself \(\lambda\)\nobreakdash-many times. In fact, if \(\tup{\alpha_\beta\mid\beta<\gamma}\) is a collection of ordinals such that \(\lambda\leq\sum_{\beta<\gamma}\alpha_\beta<\lambda^+\) then we can view \(\Add(\chi,\lambda)\) as the forcing iteration generated by \({\tup{\Add(\chi,\alpha_\beta)\mid\beta<\gamma}}\). This perspective can be helpful for calculations in forcing extensions by Cohen forcing.

\begin{prop}\label{prop:add-gives-unions}
Suppose that \(\chi\) is regular and \(\lambda\) is such that \(\chi^{{<}\chi}<\cf(\lambda)\). Then \(\Add(\chi,\lambda)\) is \(\cf(\lambda)\)\nobreakdash-final when viewed as the \(\chi\)\nobreakdash-support product \(\prod_{\alpha<\lambda}\Add(\chi,1)\).
\end{prop}

\begin{proof}
\(\chi^{{<}\chi}<\cf(\lambda)\), so certainly \(\chi<\lambda\) and hence, when viewed as the product \(\prod_{\alpha<\lambda}\Add(\chi,1)\), the forcing \(\Add(\chi,\lambda)\) is of bounded support.

Furthermore, \(\Add(\chi,\lambda)\) is \((\chi^{{<}\chi})^+\)\nobreakdash-c.c.\ so, by Proposition~\ref{prop:eg-cc}, if \(\chi^{{<}\chi}<\cf(\lambda)\) then \(\Add(\chi,\lambda)\) is \(\cf(\lambda)\)\nobreakdash-final.
\end{proof}

Note that in Proposition~\ref{prop:add-minimise} \(\lambda\) need not be a cardinal; the proof works equally well if \(\lambda\) is replaced by an ordinal and we view \(\Add(\chi,\lambda)\) as an ordinal-length product.

Proposition~\ref{prop:add-minimise} formalises the rough argument made at the beginning of Section~\ref{s:vc-forcing} to show that \(\Add(\chi,\lambda)\) can create an upper bound to \(\vc(\delta,\kappa)\) for certain values of \(\delta\) and \(\kappa\).

\begin{prop}\label{prop:add-minimise}
Let \(\lambda>\kappa\geq\chi^{{<}\chi}\) and \(\kappa\geq\delta\geq\chi\), with \(\chi\) and \(\delta\) regular. Then \(\Add(\chi,\lambda)\) forces that \(\vc(\delta^+,\kappa)=\kappa^+\).
\end{prop}

\begin{proof}
We shall relabel \(\Add(\chi,\lambda)\) to \(\Add(\chi,\lambda\oplus\kappa^+)\), where \(\lambda\oplus\kappa^+\) represents the ordinal addition of \(\lambda\) and \(\kappa^+\); we may do this because \(\lambda\geq\kappa^+\). Let \(G\) be \(V\)\nobreakdash-generic for \(\Add(\chi,\lambda\oplus\kappa^+)\) and set \(H_0\) to be the first \(\lambda\) co-ordinates of \(G\), \(H\) to be the final \(\kappa^+\) co-ordinates of \(G\), and \(W=V[H_0]\). Then \(H\) is \(W\)\nobreakdash-generic for \(\Add(\chi,\kappa^+)\), and \(V[G]=W[H]\). Note that \(\kappa^+\) is not collapsed by \(\Add(\chi,\lambda)\) due to the \((\chi^{{<}\chi})^+\)\nobreakdash-c.c.\ and \(\kappa\geq\chi^{{<}\chi}\), so \(\Add(\chi,(\kappa^+)^V)=\Add(\chi,(\kappa^+)^{V[H_0]})\). We shall show that for all \(\alpha<\kappa^+\), \((2^\kappa)^{W[H\res\alpha]}\in\sideal(\delta^+,\kappa)^{W[H]}\). This, combined with Proposition~\ref{prop:add-gives-unions}, shows that \(\Set{(2^\kappa)^{W[H\res\alpha]}\mid\alpha<\kappa^+}\) is a covering of \((2^\kappa)^{W[H]}\) by elements of \(\sideal(\delta^+,\kappa)^{W[H]}\).

Towards a contradiction, suppose that for some \(\alpha<\kappa^+\) and some \(X\in([\kappa]^\delta)^{W[H]}\), \({(2^\kappa)^{W[H\res\alpha]}\res X=(2^X)^{W[H]}}\). Then there is \(\beta<\kappa^+\), which we shall take without loss of generality to be at least \(\alpha\), such that \(X\in W[H\res\beta]\). Therefore
\begin{equation*}
(2^X)^{W[H]}=(2^\kappa)^{W[H\res\alpha]}\res X\subseteq(2^X)^{W[H\res\beta]}\subseteq(2^X)^{W[H]}
\end{equation*}
and so \((2^X)^{W[H\res\beta]}=(2^X)^{W[H]}\). However, \(\abs{X}\geq\chi\) and so there is \(x\in(2^X)^{W[H]}\) such that \(x\notin W[H\res\beta]\). For example, let \(\vphi\colon\delta\to X\) be a bijection in \(W[H]\). By Proposition~\ref{prop:strong-finality} and \(\delta\leq\kappa\), there is \(\beta'\geq\beta\) such that \(\vphi\in W[H\res\beta']\). Let \(x\) be the generic real \(\bigcup(G(\beta'))\), so \(x\colon\chi\to2\) and \(x\notin W[H\res\beta']\). Let \(y=x\concat0^{\delta\backslash\chi}\) and \(z=y\circ\vphi\). Then \(z\colon X\to2\) but \(z\notin W[H\res\beta']\) as then \(z\circ\vphi^{-1}=y\in W[H\res\beta']\) and thus \(y\res\chi=x\in W[H\res\beta']\), a contradiction. Hence, \((2^\kappa)^{W[H\res\alpha]}\in\sideal(\delta^+,\kappa)^{W[H]}\) as desired.
\end{proof}

\begin{prop}\label{prop:add-geq-lambda}
Let \(\kappa\) be regular, \(\delta<\kappa\), and \(\kappa^{{<}\kappa}<\lambda\). Then \(\Add(\kappa,\lambda)\) forces that \(\vc(\delta^+,\kappa)\geq\lambda\).
\end{prop}

\begin{proof}
The bulk of the proof rests on the following claim.

\begin{claim}\label{claim:add-geq-cf}
Let \(\kappa\) be regular, \(\delta<\kappa\), and \(\gamma\) a limit ordinal such that \(\kappa^{{<}\kappa}<\cf(\gamma)\). Then \(\Add(\kappa,\gamma)\) forces that \(\vc(\delta^+,\kappa)\geq\cf(\gamma)\).
\end{claim}

\begin{poc}
Let \(G\) be \(V\)\nobreakdash-generic for \(\Add(\kappa,\gamma)\). By Proposition~\ref{prop:add-gives-unions}, \(\Add(\kappa,\gamma)\) is \(\cf(\gamma)\)\nobreakdash-final. Since \(\delta<\kappa\) and \(\Add(\kappa,\gamma)\) is \(\kappa\)\nobreakdash-closed, \(([\kappa]^\delta)^V=([\kappa]^\delta)^{V[G]}\) and indeed \(\abs{[\kappa]^\delta}\leq\kappa^{{<}\kappa}<\cf(\gamma)\) in both \(V\) and \(V[G]\).

Note that if \(X\in([\kappa]^\delta)^{V[G]}\) then, since \(\delta\) is a cardinal in \(V[G]\), \(X\) has order type at least \(\delta\). Furthermore, if \(F\subseteq2^\kappa\) and \(F\res X=2^X\), then for all \(Y\subseteq X\), \(F\res Y=2^Y\). Hence, when we look for \(X\in[\kappa]^\delta\) such that \(F\res X=2^X\), we may without loss of generality only consider \(X\in[\kappa]^{(\delta)}\), the set of \(X\subseteq\kappa\) of order type \(\delta\). Note that \(([\kappa]^{(\delta)})^V=([\kappa]^{(\delta)})^{V[G]}\), just as in the case of \([\kappa]^\delta\). For \(X\) a set of ordinals, let \(\vphi_X\colon X\to\ot(X)\) be the unique order-preserving bijection.

Suppose, for a contradiction, that in \(V[G]\) we have \(\vc(\delta^+,\kappa)=\tau<\cf(\gamma)\), witnessed by a covering \(F=\Set{F_\alpha\mid\alpha<\tau}\). Then there is a function \(h\colon\tau\times[\kappa]^{(\delta)}\to2^\delta\) such that for each \(X\in[\kappa]^{(\delta)}\) and \(\alpha<\tau\), \(h(\alpha,X)\circ\vphi_X\notin F_\alpha\res X\). That is, \(h(\alpha,\cdot)\) acts as a witness that \(F_\alpha\in\sideal(\delta^+,\kappa)\). Replacing \([\kappa]^{(\delta)}\) by its cardinality in \(V\), say \(\abs{[\kappa]^{(\delta)}}^V=\chi<\cf(\gamma)\), we have a function \(h\colon\tau\times\chi\to V\). By \(\cf(\gamma)\)\nobreakdash-finality, there is \(\beta<\gamma\) such that \(h\in V[G\res\beta]\). We shall now show that, setting \(x\colon\kappa\to2\) to be \(\bigcup G(\beta)\), the \(\beta\)th generic subset of \(\kappa\) added by \(G\), that for all \(\alpha<\eta\) there is \(X\in[\kappa]^{(\delta)}\) such that \(x\res X=h(\alpha,X)\circ\vphi_X\). Hence, \(x\notin\bigcup F\), contradicting that \(\bigcup F=2^\kappa\).

Formally, \(x\) is the realisation of the \(\Add(\kappa,\gamma)\)\nobreakdash-name
\begin{equation*}
\ddx=\Set*{\tup{p,\tup{\check{\iota},\check{\varep}}^\bullet}\mid p\in\Add(\kappa,\gamma),p(\beta,\iota)=\varep}
\end{equation*}
so whenever \(p\in\Add(\kappa,\gamma)\) and \(\tup{\beta,\iota}\notin\supp(p)\), we can extend \(p\) to decide the value of \(\ddx(\check{\iota})\) as we desire.

Working in \(V[G\res\beta]\), let \(p\in\Add(\kappa,\gamma)/\beta=\Add(\kappa,\gamma\backslash\beta)\) and \(\alpha<\eta\). Since \(\abs{p}<\kappa\), there is \(X\in[\kappa]^{(\delta)}\) such that \(\supp(p)\cap(\Set{\beta}\times X)=\emptyset\). Extend \(p\) to \(q=p\cup\Set{\tup{\tup{\beta,\iota},h(\alpha,X)(\iota)}\mid\iota\in X}\). By density, we have that there must be some \(X\in[\kappa]^{(\delta)}\) such that \(x\res X=h(\alpha,X)\circ\vphi_X\). Since \(\alpha\) was arbitrary, this holds for all \(\alpha<\eta\) as desired.
\end{poc}
We may now conclude the proof. If \(\lambda\) is regular then the result follows directly from Claim~\ref{claim:add-geq-cf}, so suppose that \(\lambda\) is singular. Since \(\lambda\) is singular, it is a limit cardinal and, since \(\kappa^{{<}\kappa}<\lambda\), we have that \(\lambda\) is the supremum of the cardinals \(\tau\) such that \(\kappa^{{<}\kappa}<\tau<\lambda\). For each such \(\tau\), note that \(\Add(\kappa,\lambda)\) is isomorphic to \({\Add(\kappa,\lambda\oplus\tau^+)}\), and so by Claim~\ref{claim:add-geq-cf}, \(\Add(\kappa,\lambda)\) forces that \(\vc(\delta^+,\kappa)\) is at least \(\cf(\tau^+)=\tau^+\). Therefore, \(\Add(\kappa,\lambda)\) forces that \(\vc(\delta^+,\kappa)\) is at least \({\sup\Set{\tau^+\mid\kappa^{{<}\kappa}<\tau<\lambda}=\lambda}\) as required.
\end{proof}

\begin{eg}[The constellation of \(\vc(\delta,\kappa)\) for \(\kappa\leq\aleph_1\)]
We may use the tools gained to consider some possible constellations of \(\vc(\delta,\kappa)\) for \(\kappa=\aleph_0\) or \(\aleph_1\). Inspired by Figure~\ref{fig:constellation-full} and Proposition~\ref{prop:vc00-is-continuum}, the left hand side of Figure~\ref{fig:constellation-aleph1-add-prepost} shows the constellation of what we know in \(\ZFC\) about these characteristics before any forcing.

Let \(\bbP=\Add(\omega_1,\kappa)\times\Add(\omega,\lambda)\), where we choose \(\kappa\) and \(\lambda\) such that \(\1_{\bbP}\) forces that \(2^{\aleph_1}=\kappa\) and \(2^{\aleph_0}=\lambda\). That is, \(\kappa^{\omega_1^\omega}=\kappa\geq\lambda\), \(\cf(\kappa)>\omega_1\), and \(\cf(\lambda)>\omega\). Let \(G\times H\) be \(V\)\nobreakdash-generic for \(\bbP\).

We first inspect \(V[G]\), the forcing extension generated by \(\Add(\omega_1,\kappa)\). Note that in \(V[G]\) we have \(\CH\), since \(\Add(\omega_1,\kappa)\) forces that \(\omega_1^\omega=\omega_1\). Hence
\begin{equation*}
\aleph_1=\vc(\aleph_1,\aleph_0)=\vc(\aleph_0,\aleph_0)=\frakc.
\end{equation*}
Furthermore, by Proposition~\ref{prop:add-minimise} (and \(\CH\)) we have \(\vc(\aleph_2,\aleph_1)=\aleph_2\). Finally, by Proposition~\ref{prop:add-geq-lambda}, \(\vc(\aleph_1,\aleph_1)\geq\kappa\). Combining these results, we obtain the right hand side of Figure~\ref{fig:constellation-aleph1-add-prepost}.

In \(V[G][H]\), the extension generated by further forcing with \(\Add(\omega,\lambda)\) in \(V[G]\), by Proposition~\ref{prop:add-minimise} we have that \(\vc(\aleph_1,\aleph_0)=\aleph_1\) and \(\vc(\aleph_1,\aleph_1)=\vc(\aleph_2,\aleph_1)=\aleph_2\). Hence we obtain Figure~\ref{fig:constellation-aleph1-add-omega1-then-omega}. While we have the bounds \(\frakc\leq\vc(\aleph_0,\aleph_1)\leq2^{\aleph_1}\) in this model, we do not know the value that \(\vc(\aleph_0,\aleph_1)\) takes.
\end{eg}

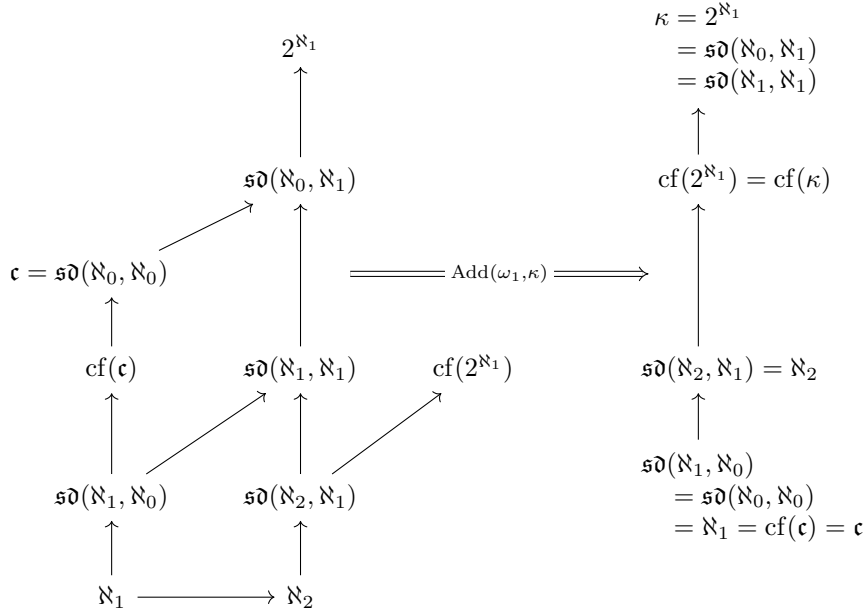
\begin{figure}
\begin{tikzcd}[column sep=2em]
&2^{\aleph_1}&&&\begin{matrix}\kappa=2^{\aleph_1}\\\onright{\kappa=}{=\vc(\aleph_0,\aleph_1)}\\\onright{\kappa=}{=\vc(\aleph_1,\aleph_1)}\end{matrix}\\
&\vc(\aleph_0,\aleph_1)\ar[u]&&&\onright{\cf(2^{\aleph_1})}{\cf(2^{\aleph_1})=\cf(\kappa)}\ar[u]\\
\mathllap{\frakc={}}\vc(\aleph_0,\aleph_0)\ar[ur]&\phantom{.}\ar[rrr,Rightarrow,"{\Add(\omega_1,\kappa)}" description,start anchor={[xshift=3ex]},end anchor={[xshift=-3ex]}]&&&\phantom{.}\\
\cf(\frakc)\ar[u]&\vc(\aleph_1,\aleph_1)\ar[uu]&\cf(2^{\aleph_1})&&\onright{\vc(\aleph_2,\aleph_1)}{\vc(\aleph_2,\aleph_1)=\aleph_2}\ar[uu]\\
\vc(\aleph_1,\aleph_0)\ar[u]\ar[ur]&\vc(\aleph_2,\aleph_1)\ar[u]\ar[ur]&&&\begin{matrix}\vc(\aleph_1,\aleph_0)\\\onright{\vc(\aleph_1,\aleph_0)}{\phantom{\vc}=\vc(\aleph_0,\aleph_0)}\\\onright{\vc(\aleph_1,\aleph_0)}{\phantom{\vc}=\aleph_1=\cf(\frakc)=\frakc}\end{matrix}\ar[u]\\
[-1em]\aleph_1\ar[r]\ar[u]&\aleph_2\ar[u]&&&
\end{tikzcd}

\caption{Constellation of \(\vc(\delta,\tau)\) for \(\tau=\aleph_0\) or \(\aleph_1\) before and after forcing with \(\Add(\omega_1,\kappa)\).}
\label{fig:constellation-aleph1-add-prepost}
\end{figure}

\begin{figure}
\begin{tikzcd}[row sep=0.5em]
&2^{\aleph_1}=\kappa\\
\mathllap{\cf(\kappa)={}}\cf(2^{\aleph_1})\ar[ur]&&\vc(\aleph_0,\aleph_1)\ar[ul]\\[1em]
&&\onright{\vc(\aleph_0,\aleph_0)}{\vc(\aleph_0,\aleph_0)=\frakc=\lambda}\ar[u]\\
\mathllap{\aleph_2=\vc(\aleph_2,\aleph_1)={}}\vc(\aleph_1,\aleph_1)\ar[uu]\ar[urr]\\
&&\onright{\cf(\frakc)}{\cf(\frakc)=\cf(\lambda)}\ar[uu]\\
\mathllap{\aleph_1={}}\vc(\aleph_1,\aleph_0)\ar[uu]\ar[urr]
\end{tikzcd}

\caption{Constellation of \(\vc(\delta,\tau)\) for \(\tau=\aleph_0\) or \(\aleph_1\) after forcing with \(\Add(\omega_1,\kappa)\times\Add(\omega,\lambda)\).}
\label{fig:constellation-aleph1-add-omega1-then-omega}
\end{figure}
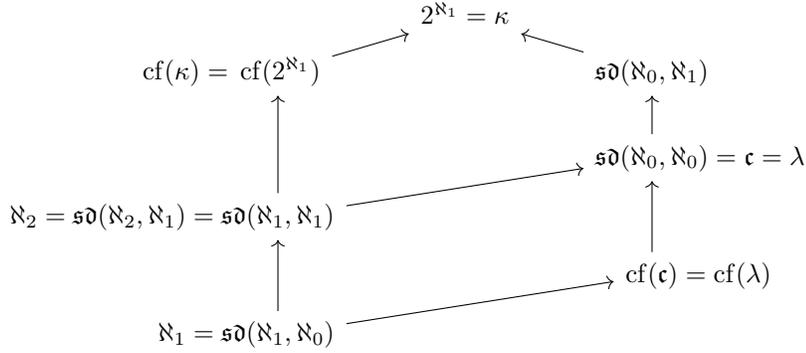

\subsection{New Set--New Function}\label{s:vc-forcing;ss:new-set-new-function}

The New Set--New Function property is introduced in \cite{cichon_combinatorial-b2_1993} as a way to explain the behaviour of large classes of forcing on \(\vc(\aleph_1,\aleph_0)\) (\(\cov(\frakP_2)\), as it is denoted in \cite{cichon_combinatorial-b2_1993}). We wish to expand upon this concept so that we can replicate the arguments of Proposition~\ref{prop:add-minimise} and Proposition~\ref{prop:add-geq-lambda}. This will complete the final half of the rough argument outline at the start of Section~\ref{s:vc-forcing}: When is \((2^\kappa)^M\in\sideal(\delta,\kappa)^N\), for \(M\subseteq N\)?

\begin{defn}
Let \(V\) be a model of \(\ZF\) and let \(\delta,\kappa\in V\) be ordinals.\footnote{Though usually one will only care about the case that \(\delta\) and \(\kappa\) are cardinals.} We say that a notion of forcing \(\bbP\in V\) has the \emph{New Set\nobreakdash--New Function property} for \(\tup{\delta,\kappa}\), written \(\NSNF(\delta,\kappa)\), if for all \(V\)-generic \(G\subseteq\bbP\), \((2^\kappa)^V\in\sideal(\delta,\kappa)^{V[G]}\). Equivalently, for all \(X\in([\kappa]^{(\delta)})^{V[G]}\) there is \(x\in(2^X)^{V[G]}\) such that for all \(y\in(2^\kappa)^V\), \(x\nsubseteq y\), so there is \(\alpha\in X\) such that \(x(\alpha)\neq y(\alpha)\).

In fact, one may define the New Set--New Function property for arbitrary nested models \(M\subseteq N\) of \(\ZF\). We say that \(M\subseteq N\) has the New Set\nobreakdash--New Function property for \(\tup{\delta,\kappa}\) if \((2^\kappa)^M\in\sideal(\delta,\kappa)^N\).
\end{defn}

Note that in \cite{cichon_combinatorial-b2_1993}, where \(\NSNF\) was originally defined, the phrase ``\(\NSNF\)'' was used to mean what we call ``\(\NSNF(\omega,\omega)\)'', while we have no analogue for what is referred to as ``\(\NSNF(\omega)\)'' in that paper.

\begin{prop}[{\cite[Lemma~4.5]{cichon_combinatorial-b2_1993}}]\label{prop:boolean-nsnf-omega-omega}
Let \(\bbB\) be a complete Boolean algebra. The following are equivalent:
\begin{enumerate}[label=\textup{\arabic*.}]
\item \(\bbB\) has the \(\NSNF(\omega,\omega)\) property.
\item For all \(\tup{u_n\mid n<\omega}\in\bbB^\omega\), if \({\prod_{n<\omega}\sum_{k>n}u_k=\1}\) then there is a decomposition \(u_n=u_n^0\lor u_n^1\) for each \(n<\omega\), where \(u_n^0\land u_n^1=\0\), such that \(\prod_{g\in2^\omega}\sum_{n<\omega}u_n^{g(n)}=\1\).
\end{enumerate}
\end{prop}

\begin{proof}
We shall only prove the \emph{only if} direction, as the other is similar.

Let \(\ddX=\Set{\tup{u_n,\check{n}}\mid n<\omega}\), and note that the condition \(\prod_{n<\omega}\sum_{k>n}u_k=\1\) is precisely saying that \(\1\forces\abs{\ddX}=\check{\omega}\). By \(\NSNF(\omega,\omega)\), there is a name \(\ddx\) for a function \(\ddX\to\check{2}\) such that for all \(y\in2^\omega\), \(\1\forces(\exists n<\omega)\ddx(n)\neq\check{y}(n)\). Let \(u_n^i=\|\ddx(\check{n})=\check{i}\|\), that is \(u_n^i=\sum\Set{u\in\bbB\mid u\forces\ddx(\check{n})=\check{i}}\), and note that in this case \(u_n=u_n^0\sqcup u_n^1\). Given \(y\in2^\omega\), the condition \(\1\forces(\exists n<\omega)\ddx(n)\neq\check{y}(n)\) is precisely saying that \(\sum_{n<\omega}u_n^{1-y(i)}=\1\). Therefore, letting \(g=1-y\), \(\sum_{n<\omega}u_n^{g(n)}=\1\) as well. Since \(y\) (and therefore \(g\)) was arbitrary, we indeed recover that \(\prod_{g\in2^\omega}\sum_{n<\omega}u_n^{g(i)}=\1\).
\end{proof}

The proof of Proposition~\ref{prop:boolean-nsnf-omega-omega} can quite easily extend to \(\NSNF(\delta,\kappa)\) for any \(\delta\leq\kappa\) if we can guarantee that \(\bbB\) will not collapse \(\delta\) or add new sets to \([\kappa]^{{<}\delta}\) (for example, if \(\bbB\) is \(\delta\)\nobreakdash-distributive).

\begin{prop}\label{prop:nsnf-kappa}
Suppose that \(\bbB\) is a \(\delta\)\nobreakdash-distributive complete Boolean algebra. Then \(\bbB\) has the \(\NSNF(\delta,\kappa)\) property if and only if for all \(\tup{u_\alpha\mid\alpha<\kappa}\in\bbB^\kappa\), if \(\prod_{A\in[\kappa]^{{<}\delta}}\sum_{\alpha\notin A}u_\alpha=\1\), then there is a decomposition \(u_\alpha=u_\alpha^0\lor u_\alpha^1\) for all \(\alpha<\kappa\), where \(u_\alpha^0\land u_\alpha^1=\0\), such that \(\prod_{g\in2^\kappa}\sum_{\alpha<\kappa}u_\alpha^{g(\alpha)}=\1\).\qed
\end{prop}

Indeed, if we wished to define an analogue of \(\NSNF\) for functions \(\kappa\to\lambda\) (rather than \(\kappa\to2\)), it is clear what the equivalent statement of Proposition~\ref{prop:nsnf-kappa} would be.

With the technology of \(\NSNF\), we can automate some arguments for decreasing \(\vc(\delta,\kappa)\).

\begin{prop}\label{prop:automatic-add-minimise}
Let \(\tup{\bbP_\alpha,\ddbbQ_\beta\mid\alpha\leq\gamma,\beta<\gamma}\) be a \(\kappa^+\)\nobreakdash-final iteration such that:
\begin{enumerate}[label=\textup{\arabic*.}]
\item For all \(\alpha<\gamma\) there is \(\alpha'\in(\alpha,\gamma)\) such that \(\1_\alpha\forces``\bbP_{\alpha'}/\alpha\) has the \(\NSNF(\delta,\kappa)\) property''; and
\item for all \(\beta<\gamma\) there is \(\beta'\in(\beta,\gamma)\) such that \(\1_\beta\forces``\bbP_{\beta'}/\beta\) adds a new function \(\delta\to2\)''.
\end{enumerate}
Then \(\1_\gamma\forces\vc(\delta^+,\kappa)\leq\cf(\gamma)\).
\end{prop}

\begin{proof}
Let \(G\) be \(V\)\nobreakdash-generic for \(\bbP_\gamma\). We shall first show that for all \(\alpha<\gamma\),
\begin{equation*}
\sideal(\delta^+,\kappa)^{V[G\res\alpha]}\subseteq\sideal(\delta^+,\kappa)^{V[G]}
\end{equation*}
Suppose otherwise, so there is \(F\in\sideal(\delta^+,\kappa)^{V[G\res\alpha]}\) such that \(F\notin\sideal(\delta^+,\kappa)^{V[G]}\) and hence there is \(X\in([\kappa]^{(\delta)})^{V[G]}\) such that \(F\res X=(2^X)^{V[G]}\). Since \(\bbP_\gamma\) is \(\kappa^+\)-final there is \(\beta<\gamma\) such that \(X\in V[G\res\beta]\). Therefore
\begin{equation*}
(2^X)^{V[G]}=F\res X\subseteq(2^X)^{V[G\res\beta]}\subseteq(2^X)^{V[G]}
\end{equation*}
and thus \((2^X)^{V[G\res\beta]}=(2^X)^{V[G]}\). Since the order type of \(X\) is \(\delta\) we similarly get that \((2^\delta)^{V[G\res\beta]}=(2^\delta)^{V[G]}\). However, this contradicts our assumption.

By assumption, for all \(\alpha<\gamma\) there is \(\alpha'<\gamma\) such that \((2^\kappa)^{V[G\res\alpha]}\) is an element of \(\sideal(\delta^+,\kappa)^{V[G\res\alpha']}\), and so \({(2^\kappa)^{V[G\res\alpha]}\in\sideal(\delta^+,\kappa)^{V[G]}}\). By \(\kappa^+\)-finality, letting \(\calC\subseteq\gamma\) be a cofinal sequence in \(\gamma\) we have that \({\Set{(2^\kappa)^{V[G\res\alpha]}\mid\alpha\in\calC}}\) is a covering of \((2^\kappa)^{V[G]}\) by elements of \(\sideal(\delta^+,\kappa)^{V[G]}\) of cardinality \(\cf(\gamma)\) as required.
\end{proof}

From the perspective of this result, we can view Proposition~\ref{prop:add-minimise} as a corollary that uses the \(\kappa^+\)\nobreakdash-finality of \(\Add(\chi,\lambda\oplus\kappa^+)\) and the \(\NSNF(\delta^+,\kappa)\) property inherent to \(\Add(\chi,1)\) for appropriate values of \(\chi\), \(\kappa\), \(\lambda\), and \(\delta\).

\subsection{Highly distributive forcing}\label{s:vc-forcing;ss:other-forcing}

We finally briefly consider two instances of forcing that change very little about \(\vc(\delta,\kappa)\) by having a high level of distributivity.

\begin{lem}
Suppose that \(\bbP\) is \((2^\kappa)^+\)\nobreakdash-distributive. Then \(\bbP\) does not change \(\vc(\delta,\kappa)\).
\end{lem}

\begin{proof}
Since \(\bbP\) is \((2^\kappa)^+\)\nobreakdash-distributive, no new subsets of \(2^\kappa\) will be added by \(\bbP\). Furthermore, no subsets of \(\kappa\) are added either, and so \(\sdim(F)\) is unchanged by \(\bbP\) for all \(F\). Finally, any covering of \(2^\kappa\) by elements of \(\vc(\delta,\kappa)\) will be of cardinality at most \(2^\kappa\) and so no new sequences of this form may be added by \(\bbP\).
\end{proof}

\begin{lem}
Suppose that \([\kappa]^{{<}\delta}<\chi\) and \(\bbP\) is \(\chi\)\nobreakdash-distributive. Setting \(\eta=\vc(\delta,\kappa)\) in the ground model, \(\bbP\) forces that \(\min\Set{\chi,\eta}\leq\vc(\delta,\kappa)\leq\eta\).

In particular, if \(\chi\leq\eta\) then \(\bbP\) does not change \(\vc(\delta,\kappa)\), and if \(\bbP\) collapses \(\eta\) to \(\chi\) then \(\bbP\) forces that \(\vc(\delta,\kappa)=\chi\).
\end{lem}

\begin{proof}
Since \(\chi>[\kappa]^{{<}\delta}\geq\kappa\), we have that \(\bbP\) does not add new elements to \(2^\kappa\). Furthermore, \(\delta<\chi\), so \(\delta\) is not collapsed and \([\kappa]^{{<}\delta}\) is similarly unchanged. Hence, if \(\calC\subseteq\sideal(\delta,\kappa)\) is a covering of \(2^\kappa\) in the ground model, it is still a covering of \(2^\kappa\) by elements of \(\sideal(\delta,\kappa)\) in the forcing extension. Therefore \(\bbP\) cannot increase \(\vc(\delta,\kappa)\).

Let \(G\) be \(V\)\nobreakdash-generic for \(\bbP\), and suppose that \(\calC=\Set{F_\alpha\mid\alpha<\gamma}\subseteq\sideal(\delta,\kappa)^{V[G]}\) is a covering of \(2^\kappa\), where \(\gamma<\chi\). Then it is sufficient to prove that \(\gamma\geq\eta=\vc(\delta,\kappa)^V\). For \(\alpha<\gamma\) and \(X\in[\kappa]^{\sdim(F_\alpha)}\), let \(h(\alpha,X)\in2^X\) be such that \(h(\alpha,X)\notin F_\alpha\res X\). Then \(h\) is a sequence of length at most \(\gamma\times[\kappa]^{{<}\delta}<\chi\), and so \(h\in V\). For \({\alpha<\gamma}\), let
\begin{equation*}
F_\alpha'=\Set{x\in(2^\kappa)^V\mid(\forall X\in[\kappa]^{\sdim(F_\alpha)})x\res X\neq h(\alpha,X)}\supseteq F_\alpha\cap V.
\end{equation*}
Then \({\Set{F_\alpha'\mid\alpha<\gamma}\in V}\) and is a covering of \(2^\kappa\) by elements of \(\sideal(\delta,\kappa)\) (witnessed by \(h\)). Hence, \(\gamma\geq\vc(\delta,\kappa)^V\) as required.
\end{proof}

\section{Relative consistencies and open questions}\label{s:future}

The table in Figure~\ref{fig:table-of-cons} details the understood relative consistency results for various cardinals or cardinal characteristics. If a box in row \(A\) and column \(B\) is labelled ``Yes'', then it is consistent with \(\ZFC\) that \(\abs{A}<\abs{B}\), and if instead it is labelled ``No'' then \(\ZFC\) proves that \(\abs{B}\leq\abs{A}\). What has been filled out so far is obtained from our example models, inspecting Figure~\ref{fig:constellation-aleph1-add-prepost} and Figure~\ref{fig:constellation-aleph1-add-omega1-then-omega}.

\begin{figure}[H]
\begin{center}
\resizebox{\hsize}{!}{
\begin{tabular}{r||c|c|c|c|c|c|c}
\(\operatorname{Con}\downarrow<\rightarrow\)&\(\vc(\aleph_1,\aleph_0)\)&\(\frakc\)&\(\cf(\frakc)\)&\(\cf\left(2^{\aleph_1}\right)\)&\(\vc(\aleph_2,\aleph_1)\)&\(\vc(\aleph_1,\aleph_1)\)&\(\vc(\aleph_0,\aleph_1)\)\\
\hline
\hline
\(\vc(\aleph_1,\aleph_0)\)&\cellcolor{cgrey}&Yes&Yes&Yes&Yes&Yes&Yes\\
\hline
\(\frakc\)&No&\cellcolor{cgrey}&No&Yes&Yes&Yes&Yes\\
\hline
\(\cf(\frakc)\)&No&Yes&\cellcolor{cgrey}&Yes&Yes&Yes&Yes\\
\hline
\(\cf\left(2^{\aleph_1}\right)\)&?&Yes&Yes&\cellcolor{cgrey}&No&Yes&Yes\\
\hline
\(\vc(\aleph_2,\aleph_1)\)&?&Yes&Yes&Yes&\cellcolor{cgrey}&Yes&Yes\\
\hline
\(\vc(\aleph_1,\aleph_1)\)&No&Yes&Yes&Yes&No&\cellcolor{cgrey}&Yes\\
\hline
\(\vc(\aleph_0,\aleph_1)\)&No&No&No&?&No&No&\cellcolor{cgrey}
\end{tabular}
}
\end{center}
\caption{Known consistencies regarding \(\vc(\delta,\kappa)\) with \(\kappa=\aleph_0\) or \(\aleph_1\).}
\label{fig:table-of-cons}
\end{figure}
While the ideals on \(2^\omega\) developed in this paper are well-understood, those developed for generalised Baire space still elude classification. A few natural questions present themselves.
\begin{qn}
What are the possible constellations of \(\vc(\delta,\kappa)\) for various values of \(\delta\) and \(\kappa\)?
\end{qn}
For example, we still have not determined if \(\vc(\aleph_0,\aleph_1)\) can be forced to be smaller than \(\cf(2^{\aleph_1})\), nor the value that it takes in our final example model (Figure~\ref{fig:constellation-aleph1-add-omega1-then-omega}).

In Section~\ref{s:vc-forcing;ss:new-set-new-function} we used the New Set--New Function technology to view Proposition~\ref{prop:add-minimise} as a corollary of Proposition~\ref{prop:automatic-add-minimise} via the \(\kappa^+\)\nobreakdash-finality of \(\Add(\chi,\lambda\oplus\kappa^+)\) and the \(\NSNF(\delta^+,\kappa)\) property of \(\Add(\chi,1)\). Proposition~\ref{prop:add-geq-lambda} is somewhat more nuanced than Proposition~\ref{prop:add-minimise}, actually using the construction of the generic object to achieve its results, but still seems ripe for generalisation.

\begin{qn}
What conditions are required (in terms of finality, \(\NSNF\), or others) to replicate the results of Proposition~\ref{prop:add-geq-lambda} for a more general class of iterations?
\end{qn}

Looking at Figure~\ref{fig:constellation-full}, we note that if \(\alpha\geq\omega\) then the sequence \({\tup{\vc(\aleph_\beta,\aleph_\alpha)\mid\beta<\alpha}}\) is an infinite non-increasing sequence of cardinals and hence takes only finitely many values between \(\aleph_{\alpha+1}\) and \(2^{\aleph_\alpha}\). Supposing that \(2^{\aleph_\alpha}\geq\aleph_{\alpha+\omega}\), the amount of different values that \(\Set{\vc(\aleph_\beta,\aleph_\alpha)\mid\beta<\alpha}\) takes could be arbitrarily large.
\begin{qn}
Are there any infinite ordinals \(\alpha\) such that \(\ZFC+2^{\aleph_\alpha}\geq\aleph_{\alpha+\omega}\) proves any bounds on the cardinality of \(\Set{\vc(\aleph_\beta,\aleph_\alpha)\mid\beta<\alpha}\)? Are there other natural conditions that may imply some bound?
\end{qn}

This framework seems to somehow be deeply involved with the two-sorted structures \((\kappa,\power(\kappa);\in)\) in a way that is intrinsically linked to how VC~dimension is understood in model theory. It seems as though there should be a general framework for abstracting other structures and dividing lines into a ``structure neutral'' way.
\begin{qn}
Does the notion of string dimension have a suitable analogue for any other dividing lines or basic structures? While classification theory is best understood in the first-order setting, this could similarly be applied to infinitary logics.
\end{qn}

The framework of finality presented in Section~\ref{s:vc-forcing;ss:finality} seems to closely follow the general literature on not adding sequences from various perspectives. For example, it is equivalent to an analogue of distributivity, and sufficient chain conditions also imply it. As such, it seems as though some notion of ``pseudoclosure'' should be able to give finality, just as a notion of forcing being closed gives rise to distributivity.
\begin{qn}
Is there an analogue to closure for forcing iterations that gives rise to finality?
\end{qn}
If so, then we will have obtained that sufficient (pseudo) chain conditions and (pseudo) closure gives rise to finality. Therefore, the next step must be to ask if ``pseudoproper'' forcings are final.
\begin{qn}
What is a ``pseudoproper'' forcing, and does it give rise to finality?
\end{qn}

Finally, we note that many of our arguments for manipulating the values of \(\vc(\delta,\kappa)\) through forcing only work when \(\delta\) is a successor. This is no coincidence, it is quite simple to show that a set is in \(\sideal(\delta^+,\kappa)\) since one only needs show that it shatters no set of cardinality \(\delta\).
\begin{qn}
How can we manipulate the values of \(\vc(\delta,\kappa)\) for \(\delta\) a limit cardinal? Are these characteristics determined by values of \(\vc(\delta',\kappa')\) for \(\delta'\) a successor instead?
\end{qn}

\section{Acknowledgements}\label{s:acknowledgements}

The author would like to thank Vincenzo Mantova for his guidance in navigating the daunting prospect of mathematical research. It would not have been possible to filter the vast corpus of mathematics and land upon interesting research without him. They would also like to thank Rosario Mennuni and Jan Dobrowolski for their encouragement and help when the author was studying the current state of affairs in this topic, and Vincenzo Mantova, Andrew Brooke-Taylor, Asaf Karagila, and Aris Papadopoulos for their feedback on early versions of this manuscript. The author would like to thank Pantelis Eleftheriou for his help in navigating the history of the independence property.

\providecommand{\bysame}{\leavevmode\hbox to3em{\hrulefill}\thinspace}
\providecommand{\MR}{\relax\ifhmode\unskip\space\fi MR }
\providecommand{\MRhref}[2]{%
  \href{http://www.ams.org/mathscinet-getitem?mr=#1}{#2}
}
\providecommand{\href}[2]{#2}


\begin{thebibliography}{10}

\bibitem{baldwin_logical_1976}
J.~T. Baldwin and Jan Saxl, \emph{Logical stability in group theory}, J.
  Austral. Math. Soc. Ser. A \textbf{21} (1976), no.~3, 267--276. \MR{407151}

\bibitem{blass_combinatorial_2010}
Andreas Blass, \emph{Combinatorial cardinal characteristics of the continuum},
  Handbook of set theory. {V}ols. 1, 2, 3, Springer, Dordrecht, 2010,
  pp.~395--489. \MR{2768685}

\bibitem{chernikov_regularity_2018}
Artem Chernikov and Sergei Starchenko, \emph{Regularity lemma for distal
  structures}, J. Eur. Math. Soc. (JEMS) \textbf{20} (2018), no.~10,
  2437--2466. \MR{3852184}

\bibitem{cichon_combinatorial-b2_1993}
J.~Cicho\'{n}, A.~Ros{\l}anowski, J.~Stepr\={a}ns, and B.~W\c{e}glorz,
  \emph{Combinatorial properties of the ideal {$\frakP_2$}}, J. Symbolic Logic
  \textbf{58} (1993), no.~1, 42--54. \MR{1217174}

\bibitem{dobrowolski_amalgamation_2023}
Jan Dobrowolski and Rosario Mennuni, \emph{The amalgamation property for
  automorphisms of ordered abelian groups}, 2023.

\bibitem{goldstern_cichon_2022}
Martin Goldstern, Jakob Kellner, Diego~A. Mej\'{\i}a, and Saharon Shelah,
  \emph{Cicho\'{n}'s maximum without large cardinals}, J. Eur. Math. Soc.
  (JEMS) \textbf{24} (2022), no.~11, 3951--3967. \MR{4493617}

\bibitem{hrushovski_groups_2008}
Ehud Hrushovski, Ya'acov Peterzil, and Anand Pillay, \emph{Groups, measures,
  and the {NIP}}, J. Amer. Math. Soc. \textbf{21} (2008), no.~2, 563--596.
  \MR{2373360}

\bibitem{simon_guide_2015}
Pierre Simon, \emph{A guide to {NIP} theories}, Lecture Notes in Logic,
  vol.~44, Association for Symbolic Logic, Chicago, IL; Cambridge Scientific
  Publishers, Cambridge, 2015. \MR{3560428}

\bibitem{van_den_dries_tame_1998}
Lou van~den Dries, \emph{Tame topology and o-minimal structures}, London
  Mathematical Society Lecture Note Series, vol. 248, Cambridge University
  Press, Cambridge, 1998. \MR{1633348}

\bibitem{vapnik_chervonenkis_class_1964}
V.~N. Vapnik and A.~Ja. \v{C}ervonenkis, \emph{On a class of algorithms for
  pattern recognition learning}, Avtomat. i Telemeh. \textbf{25} (1964),
  937--945. \MR{166008}

\bibitem{vapnik_chervonenkis_uniform_1968}
\bysame, \emph{The uniform convergence of frequencies of the appearance of
  events to their probabilities}, Dokl. Akad. Nauk SSSR \textbf{181} (1968),
  781--783. \MR{231431}

\end{thebibliography}
\end{document}